\documentclass[11pt,longbibliography]{amsart}

\usepackage{fullpage}

\usepackage{amsmath}
\usepackage{amsthm}
\usepackage{amssymb}
\usepackage{amsfonts}
\usepackage{mathtools}
\usepackage{paralist}
\usepackage[colorlinks=true,linkcolor=blue,citecolor=blue,urlcolor=blue,breaklinks]{hyperref}
\usepackage{url}
\usepackage{cases}
\usepackage{enumitem}
\usepackage{empheq}
\usepackage{multimedia}
\usepackage{longtable}
 \usepackage[foot]{amsaddr}
\usepackage[font=small,skip=2pt]{caption}

\allowdisplaybreaks
\mathtoolsset{showonlyrefs}

\numberwithin{equation}{section}

\theoremstyle{plain}
\newtheorem{thm}{Theorem}[section]

\newtheorem{lem}[thm]{Lemma}
\newtheorem{prop}[thm]{Proposition}

\theoremstyle{definition}
\newtheorem{defn}[thm]{Definition}

\theoremstyle{remark}
\newtheorem{rem}[thm]{Remark}

\theoremstyle{plain}
\newtheorem{hypo}{Hypothesis}
\makeatletter
\newcommand{\settheoremtag}[1]{
	\let\oldthehypo\thehypo
	\renewcommand{\thehypo}{#1}
	\g@addto@macro\endhypo{
	\addtocounter{hypo}{-1}
	\global\let\thehypo\oldthehypo}
}
\makeatother

\theoremstyle{remark}
\newtheorem{custom}{}
\makeatletter
\newcommand{\settag}[1]{
	\let\oldthecustom\thecustom
	\renewcommand{\thecustom}{#1}
	\g@addto@macro\endcustom{
	\addtocounter{custom}{-1}
	\global\let\thecustom\oldthecustom}
}
\makeatother

\newcommand{\be}{\begin{equation}}
\newcommand{\ee}{\end{equation}}
\newcommand{\bes}{\begin{equation*}}
\newcommand{\ees}{\end{equation*}}
\newcommand{\bfig}{\begin{figure}}
\newcommand{\efig}{\end{figure}}
\newcommand{\bt}{\begin{table}}
\newcommand{\et}{\end{table}}
\newcommand{\bc}{\begin{center}}
\newcommand{\ec}{\end{center}}
\newcommand{\ba}{\begin{array}}
\newcommand{\ea}{\end{array}}

\newcommand{\norm}[1]{\left\Vert#1\right\Vert}
\newcommand{\mt}[1]{\mathrm{#1}}
\newcommand{\wt}[1]{\widetilde{#1}}
\newcommand\restrict[1]{\raisebox{-.6ex}{$|$}_{\raisebox{.4ex}{\scriptsize $#1$}}}

\def\R{\mathbb{R}}
\def\N{\mathbb{N}}
\def\P{\mathcal{P}}
\def\Pac{\P_{\mathrm{ac},2}}
\def\A{\mathcal{A}}

\def\G{\mathcal{G}}

\def\trho{\wt{\rho}_N}

\def\grad{\nabla}
\def\div{\nabla\cdot}
\def\d{\,\mathrm{d}}
\def\p{\partial}

\def\st{\, \left|\right. \,}

\def\:{\colon}
\def\der{\mathrm{d}}
\def\E{E_N}
\def\wtE{\wt{E}_N}
\def\bx{\boldsymbol{x}}
\def\bxN{\boldsymbol{x_N}}
\def\by{\boldsymbol{y}}

\def\bz{\boldsymbol{z}}
\def\wto{\rightharpoonup}

\renewcommand{\O}{\Omega}

\DeclareMathOperator{\supp}{supp}

\DeclareMathOperator{\argmin}{argmin}

\begin{document}

\title{Convergence of a Particle Method for Diffusive Gradient Flows in One Dimension}

\author{J. A. Carrillo}
\address{Department of Mathematics, Imperial College London, South Kensington Campus, London SW7 2AZ, UK.}
\email{carrillo@imperial.ac.uk}
\author{F. S. Patacchini}
\address{Department of Mathematics, Imperial College London, South Kensington Campus, London SW7 2AZ, UK.}
\email{f.patacchini13@imperial.ac.uk}
\author{P. Sternberg}
\address{Department of Mathematics, Indiana University, Bloomington, IN 47405, Indiana, USA}
\email{sternber@indiana.edu}
\author{G. Wolansky}
\address{Mathematics Dept., Technion--Israel Institute of Technology, Haifa 32000, Israel.}
\email{gershonw@math.technion.ac.il}

\date{8 June 2017}


\maketitle

\begin{abstract}
We prove the convergence of a particle method for the approximation of diffusive gradient flows in one dimension. This method relies on the discretisation of the energy via non-overlapping balls centred at the particles and preserves the gradient flow structure at the particle level. The strategy of the proof is based on an abstract result for the convergence of curves of maximal slope in metric spaces.
\end{abstract}


\section{Introduction}
\label{sec:introduction}

In this paper we show the convergence of a particle method  to approximate the solutions to diffusion equations of the form
\begin{equation}\label{eq:pde}
	\begin{cases} \rho_t=\div\big[\rho\nabla H'(\rho(x))\big],  \quad t>0,\, x\in\O^d,\\
	\rho(0,\cdot)=\rho_0(\cdot),
	\end{cases}
\end{equation}
where $\O^d$ denotes either the closure of a bounded connected domain of $\R^d$ or all of $\R^d$ itself (when $d=1$ we simply write $\O$), $\rho(t,\cdot)\geq0$ is the unknown probability density and $\rho_0$ is a fixed element of $\P_2(\O^d)$, the set of Borel probability measures on $\O^d$ with bounded second moment---the set of Borel probability measures on $\O^d$ is simply denoted by $\P(\O^d)$. Note that we denote by the same symbol a probability measure and its density, whenever the latter exists. The function $H \: [0,\infty) \to \R$ is the \emph{density of internal energy}.

The proof of the result relies on the natural gradient flow structure of \eqref{eq:pde}; in this setting, the abstract result given by Serfaty in \cite{Serfaty} for convergence of gradient flows in metric spaces can be used. This result was in fact first proposed in \cite{SS} for the specific case of gradient flows in Hilbert spaces. The underlying metric space is given here by $\P_2(\O^d)$ and the \emph{quadratic Wasserstein distance} $d_2(\rho,\mu)$ between two measures $\rho$ and $\mu$ in $\P_2(\O^d)$, which is defined by
\be \label{eq:2-wasserstein}
	d_2(\rho,\mu) = \inf_{\gamma \in \Pi(\rho,\mu)} \left( \int_{\O^d \times \O^d} |x-y|^2 \d\gamma(x,y)\right)^{1/2},
\ee
where $\Pi(\rho,\mu)$ is the space of probability measures (also called transport plans) on $\O^d\times \O^d$ with first marginal $\rho$ and second marginal $\mu$. Note that $d_2(\rho,\mu)$ is finite for all $\rho,\mu \in \P_2(\O^d)$, and therefore the space $\P_2(\O^d)$ endowed with $d_2$ indeed defines a metric space; furthermore this space is complete, see \cite[Proposition 7.1.5]{AGS} for example. In this setting the natural \emph{continuum energy functional} $E \: \P_2(\O^d) \to \R \cup \{+\infty\}$ is
\be \label{eq:energy}
	E(\rho) = \begin{cases}
		\displaystyle \int_{\O^d} H(\rho(x)) \d x & \mbox{for all $\rho \in \Pac(\O^d)$},\\
		+\infty & \mbox{otherwise},
	\end{cases}
\ee
where $\Pac(\O^d)$ is the subset of $\P_2(\O^d)$ of probability measures which are absolutely continuous with respect to the Lebesgue measure. In this paper, the function $H$ is always either the density of internal energy for the heat equation, i.e.,
\be\label{hyp:heat}\tag{HE}
	H(x)=x\log x \quad \mbox{for all $x\in[0,\infty)$},
\ee
or a general density satisfying the following hypothesis.
\settheoremtag{(H1)}
\begin{hypo} \label{hyp:general1}
	$H$ is a proper, convex, non-negative function in $C^\infty((0,\infty)) \cap C^0([0,\infty))$ with superlinear growth at infinity and $H(0) = 0$. It also satisfies the doubling condition: there exists a constant $A>0$ such that
	\be\label{eq:doubling}
		H(x+y) \leq A(1+H(x)+H(y)) \quad \mbox{for all $x,y\in[0,\infty)$}.
	\ee
Furthermore, the function $h\: x \mapsto x^d H(x^{-d})$ is convex and non-increasing on $(0,\infty)$. 
\end{hypo}
The assumptions in \ref{hyp:general1} are typical conditions needed for the application of many theoretical results on diffusive gradient flows, which we use throughout the paper. Note that if $H$ satisfies \ref{hyp:general1}, then $E>-\infty$ since $H$ is in this case non-negative; if $H$ satisfies \eqref{hyp:heat}, then $E>-\infty$ still holds since the probability measures that we consider have finite second moments, see \eqref{eq:energy-moment}.

The assumption that $H(0) = 0$ and $h$ is convex and non-increasing implies that the energy $E$ is displacement convex, see \cite{McCann}, \cite[Section 4]{McCann2} and \cite[Theorem 5.15]{Villani} for a detailed exposition; when $d=1$, displacement convexity of $E$ is actually equivalent to convexity of $H$. Also, when $d=1$, the monotonicity condition on $h$ is a consequence of $H$ being convex and $H(0) = 0$.

In this paper we sometimes also assume the following hypothesis on $H$, in addition to \ref{hyp:general1}.
\settheoremtag{(H2)}
\begin{hypo} \label{hyp:general2}
	$H''(x) > 0$ for all $x\in(0,\infty)$ and there exists a continuous function $f:(0,\infty)\to[0,\infty)$ such that $f(1) = 1$ and
\bes
	H''(\alpha x) \geq f(\alpha) H''(x) \quad \mbox{for all $x,\alpha \in (0,\infty)$.}
\ees
\end{hypo}
Note that the density of internal energy for the heat equation satisfies all the general assumptions in \ref{hyp:general2} and \ref{hyp:general1} but the non-negativity. Also, the classical case of the porous medium equation (that is $H(x) = x^m/(m-1)$ for $m>1$) is included in the class of functions $H$ satisfying \ref{hyp:general1} and \ref{hyp:general2}, see \cite{vazquez} for a general discussion on nonlinear diffusions.

For simplicity we give now a formal way of writing \eqref{eq:pde} as a continuum gradient flow which does not require many background notions from metric spaces. The rigorous definition requires the concept of curves of maximal slope, postponed to Section \ref{subsec:continuum}. 
Let us fix a final time $T>0$. A \emph{continuum gradient flow solution} is formally defined as a curve $\rho \: [0,T] \to \P_2(\O^d)$ such that
\be \label{eq:gradient-flow}
	\begin{cases}
		\rho'(t) = - \grad_{\P_{2}(\O^d)} E(\rho(t)),\\
		\rho(0) = \rho_0,
	\end{cases}
\ee
holds in the sense of distributions on $[0,T]\times\O^d$, see \cite[Equation (8.3.8)]{AGS}. The operator $\grad_{\P_{2}(\O^d)}$ is the \emph{quadratic Wasserstein gradient} on $\P_{2}(\O^d)$, which takes the explicit form
\bes
	\grad_{\P_{2}(\O^d)} E(\rho) = -\div \left(\rho \grad \dfrac{\delta E}{\delta \rho}\right) \quad \mbox{for all $\rho \in \P_{2}(\O^d)$},
\ees
where $\delta E/\delta \rho = H'\circ\rho$ is the first variation density of $E$ at point $\rho$, and $\circ$ is the composition operator. As a by-product of the theory of gradient flows, gradient flow solutions to \eqref{eq:gradient-flow} are weak solutions to \eqref{eq:pde} up to time $T$. For theoretical issues such as existence and uniqueness of solutions to the continuum gradient flow of the form \eqref{eq:gradient-flow}, we refer the reader to \cite{JKO,Villani,AGS} and the references therein.

In this paper we approximate solutions to the continuum gradient flow \eqref{eq:gradient-flow} by finite atomic probability measures, that is by finite numbers of particles. The basic idea is to restrict the continuum gradient flow to the discrete setting of atomic measures, while keeping the gradient flow structure at the discrete level via a suitable approximation of the energy $E$ on finite numbers of Dirac masses. Given an atomic measure, we uniformly spread the mass of each point-mass in density blobs over maximal non-overlapping balls; then, we define the entropy of the atomic measures as that of these density blobs. The fact that they do not intersect allows for a fast computation of the energy and the interactions between the point-masses. This procedure was already described in the companion paper~\cite{CHPW}, where the numerical study of this method for more general gradient flows, including confinement and interaction potentials, was performed. We refer the reader to~\cite{CHPW} for a discussion about other numerical particle methods for diffusions. The goal of this paper is to show the convergence of such a discrete gradient flow to the continuum one in one dimension in the sense given in the abstract result \cite[Theorem 2]{Serfaty}, which we recall in Theorem \ref{thm:serfaty}. In order to use this result, three ``lower semi-continuity'' conditions along gradient flow solutions must be verified: one on the metric derivatives, one on the energies, and one on the slopes of the energies. For the abstract theory of the convergence of gradient flows seen as curves of maximal slope, we also refer the reader to \cite{Ortner}. Other, less abstract approaches to prove convergence of Lagrangian schemes for fourth-order equations in one dimension have been proposed in \cite{MO1,MO2}, for one-dimensional drift diffusion equations in \cite{MO3}, as well as for higher-dimensional Fokker--Planck equations in \cite{JMO}.

Our main result, Theorem \ref{thm:gradflow}, shows the convergence in the one-dimensional case with Neumann (no-flux) boundary conditions for general nonlinear diffusions (satisfying the hypotheses given above), in the case of equally-weighted particles. In general, the main difficulty that one faces with this kind of particle approximation is to characterise the subdifferentials of the discrete gradient flows. However, in one dimension we show that in our case the discrete energy is convex, allowing for an explicit, although cumbersome, characterisation of the element of minimal norm of the subdifferential. We point out that due to the choice of non-overlapping balls we have to deal with a non-smooth gradient flow at the discrete level for which we need to work with differential inclusions. Adding a confinement or potential energy to the diffusion energy \eqref{eq:energy} is of strong interest as discussed in \cite{CHPW}; in this situation, however, the computation of the element of minimal norm is not clear even in one dimension. Another difficulty is the approximation of the entropy functional; in our case, the $\Gamma$-convergence of the approximated discrete energy towards the continuum one is not difficult to show in one dimension. However, producing a good discrete energy approximation in higher dimensions is not a trivial task, see \cite{PW}.

It is worth pointing out that, as a particle method, our discretisation is mesh-free and therefore different from classical schemes for diffusion equations involving finite differences, finite volumes or finite elements. There are several motivations for studying our method. From the theoretical point of view, which is the core of this paper, it offers a rich and concrete application of the abstract result in \cite{Serfaty} on the convergence of gradient flows. From the numerical point of view, for which we refer the reader to \cite{CHPW} for more details, the method presents at least two advantages. First, it involves simpler computations of the discrete energy and its derivatives than, for example, particle methods where the mass of each particle is spread over Voronoi cells rather than over non-overlapping balls. We believe in fact that our method offers a significant numerical advantage in higher dimensions, where the derivatives of the areas, or volumes, of the Voronoi cells do not need to be computed, as already observed in \cite{CHPW}. The second advantage is the possibility of easily adding interaction and confinement potentials to the discrete energy. Although the theoretical convergence is in this case still an open question, this was numerically studied in depth in \cite{CHPW} for the case of the modified one-dimensional Keller--Segel equation for which the authors were able to show that the critical-mass properties of the equation are preserved at the discrete level.

The paper is structured as follows. In Section \ref{sec:gradflows} we give the necessary background on metric spaces to understand the proofs and discuss the notion of continuum gradient flow; we then introduce the particle method and the discrete gradient flow. Section \ref{sec:main-result-strategy} states the main result and gives the details of the strategy we follow. In Sections \ref{sec:compactness}--\ref{sec:conditions} we verify the three ``lower semi-continuity'' conditions mentioned earlier. Finally, Section \ref{sec:extension-R} discusses the possibility of extending the main result of convergence to the whole real line, i.e., with no boundary conditions, and to general weights.

\section{The gradient flows}
\label{sec:gradflows}

\subsection{Continuum gradient flow}
\label{subsec:continuum}

As already said, the gradient flow formulation given in \eqref{eq:gradient-flow} is not the one we use here, i.e., the one that allows the use of \cite[Theorem 2]{Serfaty}. Before stating the exact definition, we need to introduce a few notions from the underlying theory of gradient flows, see \cite{AGS} for a detailed account. For the sake of generality these notions are given for any complete metric space $(X,d)$. In this section, $\phi$ denotes a proper functional from $X$ to $\R \cup \{+\infty\}$ and $I$ a bounded subinterval of $\R$. We write $D(\phi)$ the \emph{domain} of $\phi$, defined by $D(\phi) = \{v\in X \st\phi(v) < +\infty\}$; the notation $D(A)$ is also used to denote the \emph{domain} of any set-valued operator $A$ from $X$ to $2^X$, that is, $D(A) = \{v\in X \st A(v) \neq \emptyset\}$. Also, let $p\geq1$.

\begin{defn}[Absolute continuity] \label{defn:absolutely-continuous}
	We say that $v \: I \to X$ is a \emph{$p$-absolutely continuous} curve if there exists $m\in L^p(I)$ such that
\be\label{eq:ac}
	d(v(t),v(\tau)) \leq \int_\tau^t m(s) \d s \quad \mbox{for all $\tau,t \in I$ with $\tau\leq t$}.
\ee
In this case we write $v \in AC^p(I,X)$, or $v \in AC(I,X)$ if $p=1$.
\end{defn}

For any $p$-absolutely continuous curve $v\: I \to X$ the \emph{metric derivative}
\bes
	|v'|_d(t) := \lim_{\tau\to t}\frac{d(v(\tau),v(t))}{|\tau - t|}
\ees
exists for almost every $t\in I$, and $|v'|_d \in L^p(I)$. In this case $|v'|_d$ satisfies \eqref{eq:ac} in place of $m$, and $|v'|_d(t) \leq m(t)$ for almost every $t\in I$ for any $m\in L^p(I)$ satisfying \eqref{eq:ac}, see \cite[Theorem 1.1.2]{AGS}.

\begin{defn}[Strong upper gradient] \label{defn:strong-upper-gradient}
	We call $g\: X \to [0,+\infty]$ a \emph{strong upper gradient} for $\phi$ if for every $v \in AC(I,X)$ we have that $g\circ v$ is a Borel function and 
\bes
	|\phi(v(t))-\phi(v(\tau))| \leq \int_{\tau}^t g(v(s)) |v'|_d(s)\d s \quad \mbox{for all $\tau,t \in I$ with $\tau\leq t$}.
\ees
\end{defn}

\begin{defn}[Local slope] \label{defn:local slope}
	We define the \emph{local slope} of $\phi$ by
\bes
	|\p \phi|(v) = \limsup_{w \to v} \dfrac{(\phi(v) - \phi(w))_+}{d(v,w)} \quad \mbox{for all $v \in D(\phi)$},
\ees 
where the subscript $+$ denotes the positive part.
\end{defn}

\begin{defn}[Curve of maximal slope] \label{defn:cms}
	Consider $g$, a strong upper gradient for $\phi$. We say that $v \in AC(I,X)$ is a \emph{$p$-curve of maximal slope} for $\phi$ with respect to $g$ if $\phi \circ v$ is almost everywhere equal to a non-increasing function $\varphi$ and
\bes
	\textstyle\varphi'(t) \leq - \frac 1p |v'|_d(t)^p - \frac 1q g(v(t))^q \quad \mbox{for almost every $t \in I$},
\ees
where $q$ is the conjugate exponent of $p$.
\end{defn}

The definition of a $p$-curve of maximal slope can be given in more generality for weak upper gradients (see \cite[Definition 1.2.2]{AGS}), rather than strong ones. However, since in this paper we only deal with strong upper gradients, we do not need such generality. 

\begin{rem}\label{rem:prop-gradflow}
When $v$ is a $p$-curve of maximal slope for a strong upper gradient $g$, we have $g\circ v |v'|_d \in L^1(I)$, $\phi \circ v \in AC(I,\R \cup \{+\infty\})$, $\phi \circ v(t) = \varphi(t)$ for all $t \in I$, and $|v'|_d(t)^p = g(v(t))^q = -\varphi'(t) = -(\phi\circ v)'(t)$ for almost every $t\in I$ (see \cite[Remark 1.3.3]{AGS}).
\end{rem}

We can now define the notion of continuum gradient flow solution.
\begin{defn}[Continuum gradient flow solution] \label{defn:continuum-gradient-flow}
	We say that $\rho \in AC^2([0,T],\P_2(\O^d))$ is a \emph{continuum gradient flow solution} with initial condition $\rho_0 \in \P_2(\O^d)$ if it is a $2$-curve of maximal slope for $E$ with respect to $|\p E|$, and if $\rho(0) = \rho_0$.
\end{defn}
The energy $E$ being displacement convex (and narrowly lower semi-continuous, see \cite[Section 10.4.3]{AGS} for instance), Definition \ref{defn:continuum-gradient-flow} makes sense since $|\p E|$ is in this case a strong upper gradient for $E$, see \cite[Corollary 2.4.10]{AGS}. 

Alternatively to \eqref{eq:gradient-flow} and Definition \ref{defn:continuum-gradient-flow}, we recall that there exists another common way of defining a continuum gradient flow, which involves the notion of subdifferential.
\begin{defn}[Subdifferential] \label{defn:subdiff}
	If $X$ is a Hilbert space with inner product $\langle\cdot,\cdot\rangle_X$ and $\phi$ is lower semi-continuous, then the \emph{subdifferential} of $\phi$ is defined, for all $x\in D(\phi)$, by
\bes
	\p \phi(x) = \left\{z\in X \,\Big|\, \liminf_{y\to x} \frac{\phi(y) - \phi(x) - \langle z,y-x\rangle_X}{|y-x|_X} \geq 0\right\}.
\ees
If $X = \P_2(\O^d)$, $Y:=(L^2_\rho(\O^d))^d$ and $\phi$ is narrowly lower semi-continuous, then we define, for all $\rho\in D(\phi)\cap\Pac(\O^d)$,
\begin{align*}
	\p \phi(\rho) = \left\{ \xi \in Y \,\Big|\, \liminf_{\nu \to \rho} \frac{\phi(\nu) - \phi(\rho) - \int_{\O^d} \langle\xi(x),t_\rho^\nu(x) - x\rangle\d\rho(x)}{d_2(\rho,\nu)} \geq 0 \right\},
\end{align*}
where $t_\rho^\nu$ is the optimal transport map from $\rho$ to $\nu$, and $\langle\cdot,\cdot\rangle$ is the classical inner product on $\R^d$. In both cases, we write $\p^0\phi(x)$ and $\p^0\phi(\rho)$ the unique elements of minimal norm of respectively $\p\phi(x)$ and $\p\phi(\rho)$, whenever they are well-defined.
\end{defn}

We can define gradient flow solutions in the following way: $\rho \in AC^2([0,T],\P_2(\O^d))$ is a continuum gradient flow solution if $\rho(t) \in D(E)\cap \Pac(\O^d)$ for almost every $t\in[0,T]$, if there exists a Borel vector field $u(t)$ such that, for almost every $t\in[0,T]$, $u(t)$ is in the tangent space of $\P_2(\O^d)$ at $\rho(t)$, $\|u(t)\|_{L_{\rho(t)}^2(\O^d)} \in L^2([0,T])$,
\bes
	u(t) \in -\p E(\rho(t)),
\ees
and the continuity equation
\bes
	\rho'(t) + \div (\rho(t)u(t)) = 0
\ees
holds in the sense of distributions on $[0,T]\times\O^d$. Since $E$ is displacement convex and lower semi-continuous, we have existence and uniqueness of such gradient flows. Moreover, this notion of continuum gradient flows and Definition \ref{defn:continuum-gradient-flow} are equivalent; and in this case, the \emph{velocity field} $u(t) = -\p^0 E(\rho(t)) = - \grad (\delta E/\delta \rho)(t)$ exists for almost every $t\in[0,T]$, see \cite[Theorems 5.3, 5.5 and 5.8]{Ambrosio} for more details. In the following we only work with Definition \ref{defn:continuum-gradient-flow}.

\subsection{Particle method and discrete gradient flow}
\label{subsec:particle-method}
For the rest of the paper we restrict ourselves to the \emph{one-dimensional} case ($d=1$). Discussions on possible extensions to higher dimensions are given throughout the text. 

We describe now the particle method which is used to approximate the continuum gradient flow. In this method, the underlying probability measure is characterised by the particles' positions $(x_1,\dots,x_N)\in \O^N$ and the associated equal weights $w:=\left(1/N,\dots,1/N\right) \in (0,1)^N$, where $N \geq 2$ is the total number of particles considered. Throughout this paper, the positions $(x_1,\dots,x_N)$ are evolving in time but the weights $w$ are fixed. Also, we denote by $\O_w^N$ the space of particles with weights $w$, that is, $\bxN:= (x_1,\dots,x_N) \in \O_{w}^N$ means that each particle $x_i$ is in $\O$ and is associated with the weight $1/N$. Notice the boldface font when referring to elements of $\O_w^N$.

By convention, in the rest of the paper, whenever particles $\bxN\in \O_w^N$ are considered, they are assumed to be distinct and sorted increasingly, i.e., $x_{i+1} >  x_i$ for all $i \in \{1,\dots,N-1\}$.

The most natural representation of the underlying probability measure is the empirical measure
\bes
	\bxN \mapsto \mu_N = \frac1N \sum_{i=1}^N \delta_{x_{i}},
\ees
which belongs to the space of \emph{atomic measures}
\bes
	\A_{N,w}(\O) := \left\{\mu \in \P_2(\O) \,\Big|\, \exists\, \bxN \in \O_w^N, \mu = \frac1N\sum_{i = 1}^N \delta_{x_{i}} \right\}.
\ees

\begin{defn}[Inter-particle distance] \label{defn:interparticles}
	For any particles $\bx = (x_1,\dots,x_N) \in \O_w^N$ we denote the \emph{inter-particle distance} by the positive quantity (eventually $+\infty$ by convention)
\bes
	\Delta x_i := x_i-x_{i-1} \quad \mbox{for $i \in \{1,\dots,N+1\}$},
\ees
with the conventions for $x_0$ and $x_{N+1}$ given in \eqref{eq:convention1} and \eqref{eq:convention2} according to the boundary conditions considered. Furthermore, for any $i \in \{2,\dots,N\}$, the interval $[x_{i-1},x_i]$ is called the \emph{inter-particle interval}. We also write
\be\label{eq:ri}
	r_i = \min(\Delta x_i,\Delta x_{i+1}) \quad \mbox{for all $i \in \{1,\dots,N\}$}.
\ee
\end{defn}

\begin{defn}[Discrete energy] \label{defn:discrete-energy}
	We define the \emph{discrete energy} $\E \: \A_{N,w}(\O) \to \R$, for all $\mu_N \in \A_{N,w}(\O)$ with particles $\bxN \in \O_w^N$, by
\be \label{eq:energy-discrete}
	E_N(\mu_N) = \sum_{i=1}^N |B_i| H\left(\frac{1}{N|B_i|}\right) = \frac1N \sum_{i=1}^N h(N|B_i|) =  \frac1N \sum_{i=1}^N h(Nr_i),
\ee
where $h$ is as in \ref{hyp:general1}, $B_i := B_{r_i/2}(x_i)$ with $r_i$ given in \eqref{eq:ri}, and $|B_i|= r_i$.
\end{defn}
Note that $\E$ is finite over the whole $\A_{N,w}(\O)$ since $H$ is pointwise finite. The essence of this discrete approximation lies in the adequate treatment of the energy $E$, which becomes infinity on point-masses; here the mass of each particle is uniformly spread to circumvent this problem. To this end, consider
\be \label{eq:density-on-balls}
	\bxN \mapsto \rho_N = \frac1N\sum_{i = 1}^N\frac{\chi_{B_{i}}}{|B_{i}|},
\ee
where $\chi_{B_{i}}$ is the characteristic function of $B_{i}$. Clearly $\rho_N$ is in $\Pac(\R)$, and thus the energy \eqref{eq:energy} integrated on $\R$ is well-defined for $\rho_N$. An example of what $\rho_N$ looks like is given in Figure \ref{fig:rhoN}.
\bfig[!ht]
\centering
	\includegraphics[scale=1]{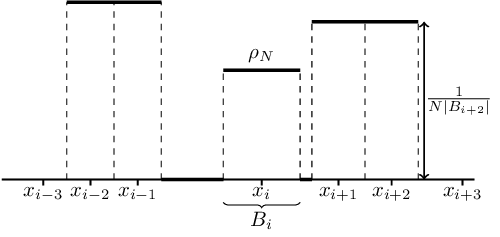}
\caption{The reconstructed piecewise constant density $\rho_N$ \label{fig:rhoN}}
\efig
The representation $\rho_N$ does not involve overlapping of balls, but involves ``gaps" between balls whose sizes are intuitively expected to decrease as the number of particles increases. More rigorously, we actually prove in Lemmas \ref{lem:bounds-distances} and \ref{lem:properties-interparticle} and in Remark \ref{rem:gaps} that, if no-flux boundary conditions are considered, gaps individually decrease like $1/N$ and their sum tends to 0 as $N$ increases. By plugging \eqref{eq:density-on-balls} into the energy \eqref{eq:energy} integrated on $\R$, one gets that $E_N$ defined above is exactly
\bes
	\E(\mu_N) = E(\rho_N) = \int_{\R} H(\rho_N(x)) \d x.
\ees
This choice of non-overlapping particles has the main advantage of reducing the computational cost of the discrete energy functional and its subdifferential. Note that $\rho_N$ may not belong to $\P_\mt{ac,2}(\O)$ if $\O$ is bounded; indeed the discretisation balls $B_1$ and $B_N$ may not be contained in $\O$ in this case.

Since the expression above depends essentially on $\bxN \in \O_w^N$, we can define the discrete energy equivalently as a function of $\bxN \in \O_w^N$ instead of $\mu_N \in \A_{N,w}(\O)$:
\be \label{eq:energy-discrete-particles}
	\wt{E}_N(\bxN) := \E(\mu_N) \quad \mbox{for all $\mu_N \in \A_{N,w}(\O)$ with particles $\bxN \in \O_w^N$}.
\ee

We give now the two possible boundary conditions we consider, depending on whether $\O$ is $\R$ or the closure of a bounded connected subset of $\R$ (with no loss of generality we take a closed ball).

\settag{\underline{\smash{Discretisation in $\R$: no boundary conditions}}}
\begin{custom}
When $\O = \R$ we define two fictitious particles 
\be\label{eq:convention1}
	x_{N+1} = -x_0 = +\infty,
\ee
so that $r_1= \Delta x_2$ and $r_N=\Delta x_N$.
\end{custom}

\settag{\underline{\smash{Discretisation in a closed ball of $\R$: no-flux boundary conditions}}}
\begin{custom}
Fix $\ell\in(0,\infty)$. When $\O = [-\ell,\ell]$ we define two fictitious particles
\be\label{eq:convention2}
	\begin{cases} x_0 = -\infty & \mbox{if $x_1 \in (-\ell,\ell)$},\\ x_0 = -2\ell -x_2 & \mbox{if $x_1 \in (-2\ell-x_2, -\ell]$}, \end{cases} \quad \begin{cases} x_{N+1} = +\infty & \mbox{if $x_N \in (-\ell,\ell)$},\\ x_{N+1} =  2\ell -x_{N-1} & \mbox{if $x_N \in [\ell,2\ell-x_{N-1})$}, \end{cases}
\ee
so that if $x_1 = -\ell$, then $r_1 = \Delta x_2 = \Delta x_1$, and similarly for $x_N$ and $r_N$.
\end{custom}

	The difference between the discretisation in $\R$ and that in $[-\ell,\ell]$ lies in the treatment of the end particles $x_1$ and $x_N$. When $\O=\R$ there are no boundary conditions to consider and therefore no restrictions on where the particles flowing according to the discrete system \eqref{eq:discrete-gradient-flow-inclusion} can move; this is allowed by the fact that the two fictitious particles are placed at infinity and therefore have no influence on the evolution of the ``real'' particles. When $\O=[-\ell,\ell]$, however, particles cannot go out of the domain; this is ensured by the presence of the two fictitious particles in \eqref{eq:convention2}. Indeed, these particles have no influence on the ``real'' ones as long as these stay contained in $(-\ell,\ell)$; when a ``real'' particle reaches the boundary of the domain, however, the fictitious particles ensure that it stays there, without having an influence on the other ``real" particles, see Lemma \ref{lem:boundary-particles-fixed}. 

Now that we have a discrete setting, we can define the discrete analogue of a continuum gradient flow solution given in Definition \ref{defn:continuum-gradient-flow}. 

\begin{defn}[Discrete gradient flow solution] \label{defn:discrete-gradient-flow}
	We say that $\mu_N \in AC^2([0,T],\A_{N,w}(\O))$ is a \emph{discrete gradient flow solution} with initial condition $\mu_N^0 \in \A_{N,w}(\O)$ if it is a $2$-curve of maximal slope for $\E$ with respect to $|\p \E|$, and if $\mu_N(0) = \mu_N^0$.
\end{defn}

Equivalently, by \eqref{eq:energy-discrete-particles}, the discrete gradient flow can be defined on $\O_w^{N}$ rather than $\A_{N,w}(\O)$. 
\begin{defn}[Discrete gradient flow solution for particles] \label{defn:discrete-gradient-flow-particles}
	We say that $\bxN \in AC^2([0,T],\O_w^N)$ is a \emph{discrete gradient flow solution (for particles)} with initial condition $\boldsymbol{x_N^0} \in \O_w^N$ if it is a $2$-curve of maximal slope for $\wtE$ with respect to $|\p \wt{E}_N|$, and if $\bxN(0) = \boldsymbol{x_N^0}$.
\end{defn}

These two formulations being equivalent, we use them interchangeably in the rest of the paper.

\begin{rem}\label{rem:upper-gradients}
	Definitions \ref{defn:discrete-gradient-flow} and \ref{defn:discrete-gradient-flow-particles} make sense since, by the proof of Proposition \ref{prop:discrete-gradient-flow-well-posed}, $\E$ and $\wtE$ are convex and lower semi-continuous, which makes sure that the respective local slopes are strong upper gradients, see \cite[Proposition 1.4.4]{AGS}.
\end{rem}

Definition \ref{defn:discrete-gradient-flow-particles} can be reformulated by means of a differential inclusion, see \cite[Proposition 1.4.1]{AGS}.

\begin{prop} \label{prop:differential-inclusion0}
	A curve $\bxN \in AC^2([0,T],\O_w^{N})$ is a discrete gradient flow solution with initial condition $\boldsymbol{x_N^0} \in \O_w^{N}$ if and only if it satisfies
\bes
	\begin{cases}
		\frac1N\bxN'(t) \in - \p \wt{E}_N(\bxN(t)) & \mbox{for almost all $t \in (0,T]$},\\
		\bxN(0) = \boldsymbol{x_N^0},
	\end{cases}
\ees
where $\bxN'$ is the speed of the curve $\bxN$.
\end{prop}

The presence of a differential inclusion in Proposition \ref{prop:differential-inclusion0} comes from the fact that the gradient of the discrete energy $\wt{E}_N$ is not everywhere defined since it involves the minimum function. 

The formulation given in Proposition \ref{prop:differential-inclusion0} is not a standard differential inclusion because of the presence of the weights $1/N$ in the left-hand side. To cope with this, we introduce the following inner product on $\O_w^{N}$.
\begin{defn}[Weighted inner product on $\O_w^{N}$] \label{defn:weighted-ip}
	For all $\bx,\by\in\O_w^{N}$ we define the \emph{weighted inner product} between $\bx = (x_1,\dots,x_N)$ and $\by = (y_1,\dots,y_N)$ as
\bes
	\langle \bx,\by \rangle_w = \frac1N \sum_{i=1}^N \langle x_i,y_i\rangle = \frac1N \sum_{i=1}^N x_iy_i.
\ees
\end{defn}
From now on, the Euclidean space $\O_w^{N}$ is endowed with this inner product. This definition clearly induces the following weighted norm on $\O_w^{N}$.
\bes
	|\bx|_w := \sqrt{\langle \bx,\bx \rangle_w} = \sqrt{\textstyle{\frac1N\sum}_{i=1}^N |x_i|^2} \quad \mbox{for all $\bx\in\O_w^{N}$}.
\ees
It also induces a subdifferential structure: for any functional $\phi\:\O_w^{N} \to \R$,
\bes
	\p_w \phi(\bx) := \left\{\bz\in\O_w^{N} \mid \textstyle{\frac1N}\bz \in \p \phi(\bx)\right\} \quad \mbox{for all $\bx\in\O_w^{N}$},
\ees
and we can then define the element $\partial_w^0 \phi(\bx)$ with minimal norm accordingly. We can now rewrite Proposition \ref{prop:differential-inclusion0} as follows.
\begin{prop} \label{prop:differential-inclusion}
	A curve $\bxN \in AC^2([0,T],\O_w^{N})$ is a discrete gradient flow solution with initial condition $\boldsymbol{x_N^0} \in \O_w^{N}$ if and only if it satisfies
\be \label{eq:discrete-gradient-flow-inclusion}
	\begin{cases}
		\bxN'(t) \in - \p_w \wt{E}_N(\bxN(t)) & \mbox{for almost all $t \in (0,T]$},\\
		\bxN(0) = \boldsymbol{x_N^0}.
	\end{cases}
\ee
\end{prop}

The proposition below shows that, in dimension one, the gradient flow inclusion \eqref{eq:discrete-gradient-flow-inclusion} is well-posed, that is, it has one and only one solution.
\begin{prop} \label{prop:discrete-gradient-flow-well-posed}
	The discrete gradient flow inclusion \eqref{eq:discrete-gradient-flow-inclusion} is well-posed. Furthermore, the solution $\bxN$ satisfies $\bxN'(t) = -\p_w^0 \wtE(\bxN(t))$ for almost every $t \in [0,T]$.
\end{prop}
\begin{proof}
	We use here a result from the theory of maximal monotone operators (see \cite{Brezis,Aubin} for example). The precise result we use is \cite[Theorem 1 of Section 3.2 and Proposition 1 of Section 3.4]{Aubin}, which states that if $\wt{E}_N \: \O_w^N \to \R \cup \{+\infty\}$ is proper, lower semi-continuous and convex, then the gradient flow inclusion \eqref{eq:discrete-gradient-flow-inclusion} has a unique solution if $\boldsymbol{x_N^0} \in D(\p_w\wtE) = \O_w^N$; furthermore, the solution $\bxN$ to \eqref{eq:discrete-gradient-flow-inclusion} is such that $\bxN'(t) = -\p_w^0 \wtE(\bxN(t))$ for almost every $t \in [0,T]$. Trivially $\wt{E}_N$ is proper and lower semi-continuous since $\min$ is continuous on $\R^2$ and $h\: x\mapsto xH(x^{-1})$ is continuous on $(0,\infty)$. We are left with showing the convexity of $\wt{E}_N$. 

Let $\lambda \in [0,1]$ and $\bx,\by \in \O_w^N$. Write, for all $a,b\in\R$, $[a,b]_\lambda = \lambda a + (1-\lambda)b$, and $r_i(\bx) = \min(\Delta x_i,\Delta x_{i+1})$ and $r_i(\by) = \min(\Delta y_i,\Delta y_{i+1})$. Since $\min$ is concave on $\R^2$ and $h$ is non-increasing and convex on $(0,\infty)$, $h\circ\min$ is convex on $[0,\infty)^2$. Then \eqref{eq:energy-discrete} gives the convexity of $\wt{E}_N$:
\begin{align*}
	\wt{E}_N(\lambda \bx + (1-\lambda)\by) &=\displaystyle \dfrac 1N \sum_{i = 1}^N h\left(N\min([\Delta x_i,\Delta y_i]_\lambda,[\Delta x_{i+1},\Delta y_{i+1}]_\lambda)\right)\\
	& \leq \dfrac{\lambda}{N} \sum_{i = 1}^N h(Nr_i(\bx)) + \dfrac{1 - \lambda}{N} \sum_{i = 1}^N h(Nr_i(\by)) = \lambda\wt{E}_N(\bx) + (1-\lambda)\wt{E}_N(\by).
\end{align*}

Note that once the convexity of $\wtE$ is shown the well-posedness of \eqref{eq:discrete-gradient-flow-inclusion} does not only follow from monotone operator theory but also from standard gradient flow theory, see \cite[Section 11.1]{AGS}.
\end{proof}

\begin{rem}\label{rem:issues}
The extension of Section \ref{subsec:particle-method} to higher dimensions presents two main issues. The first comes from the treatment of the boundary conditions. Ensuring no-flux boundary conditions in higher dimensions is common practice in the numerics of sweeping processes, where the velocity field of the considered discrete gradient flow is projected onto the tangent plane to the domain whenever a particle is on the boundary of this domain (see \cite{CSW,Venel,ET1,ET2} and the references therein for a detailed account). When $d=1$ the projection of the velocity of a particle exiting the domain onto the tangent plane of this domain is 0, which corresponds indeed to adding the two fictitious particles \eqref{eq:convention2}. The second issue is the well-posedness of the discrete gradient flow. At the continuum level we know that the energy $E$ is displacement convex in any dimension $d\geq1$. Unfortunately, we are unable to prove, or disprove, that this property is preserved at the discrete level for $d > 1$; for $d=1$, this is shown in the proof of Proposition \ref{prop:discrete-gradient-flow-well-posed}. This lack of convexity also makes it unsure that the discrete local slopes are actually strong upper gradients, see Remarks \ref{rem:upper-gradients} and \ref{rem:dim}.
\end{rem}

\section{Main result and strategy}
\label{sec:main-result-strategy}

Before stating the main result, Theorem \ref{thm:gradflow}, we introduce some notations and definitions.

\begin{defn}[Smooth set] \label{defn:smooth}
	We define the subset $\G(\O)$ of $\Pac(\O)$ as follows. We write $\rho \in \G(\O)$ if there exists $r > 0$ such that all the items below hold.
\begin{enumerate}[label=(\roman*)]
	\item $\supp\rho = [-r,r]$,
	\item $\rho\restrict{\supp\rho} \in C^1(\supp\rho)$,
	\item $\min_{\supp\rho}\rho > 0$,
	\item if $\O = [-\ell,\ell]$, then $r = \ell$.
\end{enumerate}
\end{defn}

\begin{rem} \label{rem:g-finite}
	Any $\rho \in \G(\O)$ satisfies $E(\rho) < +\infty$.
\end{rem}

\begin{defn}[Recovery sequence and well-preparedness]\label{defn:initial-set}
	Let $\rho \in \P_2(\O)$. Any $(\mu_N)_{N\geq2}$ with $\mu_N \in \A_{N,w}(\O)$ for all $N\geq2$ such that $\mu_N\wto\rho$ narrowly as $N \to \infty$ and $\limsup _{N\to \infty} \E(\mu_N) \leq E(\rho)$ is said to be a \emph{recovery sequence} for $\rho$. 

Let $(\bxN)_{N\geq2}$ be the particles of $(\mu_N)_{N\geq 2}$. We say that $(\mu_N)_{N\geq 2}$ is \emph{well-prepared} for $\rho$ if it is a recovery sequence for $\rho$ and there exist $a_1,a_2>0$ such that $a_1/N \leq \Delta x_i \leq a_2/N$ for all $i\in\{2,\dots,N\}$ and all $N\geq2$; if $\rho \in \G(\O)$, we moreover require $x_N = -x_1 = r$.
\end{defn}

An example of a well-prepared sequence for any $\rho\in\G(\O)$ is given in Lemma \ref{lem:limsup-continuous}. We can now state our main result.

\begin{thm}[Main theorem]\label{thm:gradflow}
	Let $H$ be given by \eqref{hyp:heat} or let it satisfy \ref{hyp:general1}. Suppose that $\mu_N \in AC^2([0,T],\A_{N,w}(\O))$, with particles $\bxN \in AC^2([0,T],\O_w^N)$, is a discrete gradient flow solution with initial condition $\mu_N^0 \in \A_{N,w}(\O)$, with particles $\boldsymbol{x_N^0} \in \O_w^N$. Let $\rho_0 \in \G(\O)$ and assume that $(\mu_N^0)_{N\geq2}$ is well-prepared for $\rho_0$ according to Definition \ref{defn:initial-set}. Then there exists $\rho\in AC^2([0,T],\P_2(\O))$ such that $\mu_N(t)\wto \rho(t)$ narrowly as $N\to\infty$ for all $t \in [0,T]$. Moreover, if $\O = [-\ell,\ell]$ and $H$ satisfies \ref{hyp:general2}, then $\rho$ is the continuum gradient flow solution associated to \eqref{eq:pde} in the sense of Definition \ref{defn:continuum-gradient-flow}, and
\be \label{eq:limits}
	\begin{cases}
		\displaystyle\lim_{N\to\infty} |\mu_N'|_{d_2} = |\rho'|_{d_2} \quad \mbox{in $L^2([0,T])$},\\
		\displaystyle \lim_{N\to\infty} \E(\mu_N(t)) = E(\rho(t)) \quad \mbox{for all $t \in [0,T]$},\\
		\displaystyle\lim_{N\to\infty}|\p \E(\mu_N)| = |\p E(\rho)| \quad \mbox{in $L^2([0,T])$}.
	\end{cases}
\ee 
\end{thm}

\begin{rem} \label{rem:stability-initial}
	In Theorem \ref{thm:gradflow} it can actually be proved that the convergence of $\mu_N(t)$ to $\rho(t)$ is stronger than narrow; it is indeed in $d_p$ for any $1\leq p <2$, where $d_p$ is the $p$th Wasserstein distance defined analogously to \eqref{eq:2-wasserstein}, see the proof of Lemma \ref{lem:compactness}. Obviously, when $\O=[-\ell,\ell]$ we actually have $d_2$-convergence since narrow and $d_2$-convergences are then equivalent. 

	Suppose that we had the convergence of the gradient flow regardless of $\O$ being $[-\ell,\ell]$ or $\R$---see Theorem \ref{thm:gradflow-R} for an attempt at such a generalisation. Then we make the following remark. In our main theorem we assume some regularity on the initial datum: $\rho_0 \in \G(\O)$. If we want to start with a general $\rho_0 \in \P_2(\O)$, then we can use the stability property of the initial conditions with respect to $d_2$; that is, if $\rho_1$ and $\rho_2$ are two continuum gradient flow solutions in $AC^2([0,T],\P_2(\O))$ with respective initial conditions $\rho_1^0$ and $\rho_2^0$ in $\P_2(\O)$, then
\be \label{eq:stability-initial}
	d_2(\rho_1(t),\rho_2(t)) \leq d_2(\rho_1^0,\rho_2^0) \quad \mbox{for all $t \in [0,T]$},
\ee
see \cite[Theorem 5.5]{Ambrosio}. Let $1\leq p <2$ and consider $\rho_0\in\P_2(\O)$ and $\rho_0^\delta\in\G(\O)$ for all $\delta > 0$, two initial data such that $d_2(\rho_0,\rho_0^\delta) \to 0$ as $\delta \to 0$. Assume that $\mu_{N}^\delta$ is a discrete gradient flow solution which is well-prepared initially for $\rho_0^\delta$. Then, by Theorem \ref{thm:gradflow} and what observed above, the continuum gradient flow solution $\rho_\delta \in AC^2([0,T],\P_2(\O))$ emanating from $\rho_0^\delta$ is such that $d_p(\mu_{N}^\delta(t),\rho_\delta(t)) \to 0$ as $N \to \infty$ for all $t\in[0,T]$. Let the continuum gradient flow solution emanating from $\rho_0$ be $\rho \in AC^2([0,T],\P_2(\O))$. Now, by the triangular inequality, the non-decreasing monotonicity of the sequence $(d_p)_{p\geq1}$ and \eqref{eq:stability-initial}, for all $t \in [0,T]$ we have
\bes
	d_p(\mu_{N}^\delta(t),\rho(t)) \leq d_p(\mu_{N}^\delta(t),\rho_\delta(t)) + d_p(\rho_\delta(t),\rho(t)) \leq d_p(\mu_{N}^\delta(t),\rho_\delta(t)) + d_2(\rho_0^\delta,\rho_0).
\ees
Thus, for all $\delta > 0$, there exists $N(\delta)\geq2$ such that $d_p(\mu_{N(\delta)}^\delta(t),\rho(t)) \leq  \delta+ d_2(\rho_0^\delta,\rho_0)$. Then
\bes
	\lim_{\delta\to0} d_p(\mu_{N(\delta)}^\delta(t),\rho(t)) = 0.
\ees
Hence the continuum gradient flow $\rho$ is well approximated by the subsequence $(\mu_{N(\delta)}^\delta)_{\delta>0}$ as $\delta\to0$.
\end{rem}

To prove Theorem \ref{thm:gradflow}, we want to use \cite[Theorem 2]{Serfaty}, which we state below in our context.
\begin{thm}[\cite{Serfaty}]\label{thm:serfaty}
	Let $\mu_N \in AC^2([0,T],\A_{N,w}(\O))$ be a discrete gradient flow according to Definition \ref{defn:discrete-gradient-flow}. Assume that $\mu_N(t) \wto \rho(t)$ narrowly as $N\to\infty$ for all $t \in [0,T]$ for some $\rho\in AC^2([0,T],\P_2(\O))$. Suppose furthermore that $(\mu_N(0))_{N\geq2}$ is a recovery sequence for $\rho(0)$ according to Definition \ref{defn:initial-set}, and that the following conditions hold for all $t\in[0,T]$.
\begin{enumerate}[label=(C\arabic*)]
	\item \label{cond:md} $\displaystyle \liminf_{N\to \infty} \int_0^t |\mu_N'|_{d_2}(s)^2 \d s \geq \int_0^t|\rho'|_{d_2}(s)^2\d s$.\\
	\item \label{cond:liminf} $\displaystyle \liminf_{N\to\infty} \E(\mu_N(t)) \geq \displaystyle E(\rho(t))$.\\
	\item \label{cond:slopes} $\displaystyle \liminf_{N\to \infty} |\p\E|(\mu_N(t)) \geq |\p E|(\rho(t))$.
\end{enumerate}
Then $\rho$ is a continuum gradient flow according to Definition \ref{defn:continuum-gradient-flow}, and \eqref{eq:limits} holds.
\end{thm}

\begin{rem}\label{rem:dim}
	As stated in \cite{Serfaty} for all $d\geq1$, Theorem \ref{thm:serfaty} actually requires $|\p E|$ and $|\p\E|$ to be strong upper gradients; as already discussed below Definition \ref{defn:continuum-gradient-flow} and in Remark \ref{rem:upper-gradients}, this is surely the case for $|\p E|$ in any dimension and for $|\p \E|$ when $d=1$. If $d>1$, it is less clear for $|\p\E|$ since we are not able to prove, or disprove, the convexity of $\E$. This, along with those mentioned in Remark \ref{rem:issues}, is one of the reasons why we restrict ourselves to $d=1$ for the proof of Theorem \ref{thm:gradflow}. Other reasons are the simplifications induced by the natural increasing ordering of particles, and the possibility of computing the subdifferential explicitly, see Lemmas \ref{lem:subdiff-EN} and \ref{lem:cases}.
\end{rem}

The result of Theorem \ref{thm:gradflow} has two main parts: the compactness part, which shows the existence of the limiting $\rho$, and the convergence part, which shows that this $\rho$ is indeed the continuum gradient flow solution. The proof of the second part entirely relies on Theorem \ref{thm:serfaty}, and therefore reduces to showing \ref{cond:md}--\ref{cond:slopes}. In order, we first show the compactness part of the result and \ref{cond:md}, and then \ref{cond:liminf} and \ref{cond:slopes}. Let us remark that the condition that $H$ satisfies \ref{hyp:general2} in the main theorem is actually only needed in the proof of \ref{cond:slopes}.
	
The restriction of the convergence part to $\O = [-\ell,\ell]$ stems from the difficulty of treating the gaps between inter-particle intervals when $\O = \R$. The conditions \ref{cond:md} and \ref{cond:liminf} are actually shown for $\O = \R$ as well, whereas the condition \ref{cond:slopes} is the one that requires $\O = [-\ell,\ell]$. The possibility of extending the proof of \ref{cond:slopes} to no boundary conditions is discussed in Section \ref{sec:extension-R}.


\section{Condition on the metric derivatives and compactness result}
\label{sec:compactness}

We justify the existence of the limiting $\rho$ of Theorem \ref{thm:gradflow} and show \ref{cond:md}. To this end we first give in Lemma \ref{lem:carleman} two Carleman-type estimates relating the continuum energy and the second moment. An estimate similar to \eqref{eq:energy-l1} can be found in \cite[Lemma 2.2]{Blanchet} and \cite{JKO}. We denote by $M_2(\rho) := \int_{\O} |x|^2 \d\rho(x)$ the second moment of $\rho$, for any $\rho \in \P_2(\O)$. Note first that Lemmas \ref{lem:carleman} and \ref{lem:bounds-moment-energy} are only stated for $\O=\R$ because their proofs are much easier if $\O = [-\ell,\ell]$. Indeed, in this case Lemma \ref{lem:carleman} comes from the convexity of $H$ and the use of Jensen's inequality to get
\bes
	E(\rho) \geq h(2\ell) \quad \mbox{if $\rho \in \P_2([-\ell,\ell])$}, \quad \norm{H\circ \rho}_{L^1([-\ell,\ell])} \leq E(\rho) + 4\ell \quad \mbox{if $\rho \in \Pac([-\ell,\ell])$},
\ees
and Lemma \ref{lem:bounds-moment-energy} is trivial.

\begin{lem}\label{lem:carleman}
	Let $\O = \R$, and let $H$ satisfy \ref{hyp:general1} or be as in \eqref{hyp:heat}. For all $\delta > 0$ and $\rho \in \P_2(\R)$,
	\be \label{eq:energy-moment}
		E(\rho) \geq -K_\delta - \delta M_2(\rho), 
	\ee
	where $K_\delta := \sqrt{2\pi/\delta}$, and, if $\rho \in \Pac(\R)$, 
	\be \label{eq:energy-l1}
		\norm{H\circ \rho}_{L^1(\R)} \leq E(\rho) + 2\delta M_2(\rho) + 2K_\delta.
	\ee
\end{lem}
\begin{proof}
Let $\delta>0$. If $H$ satisfies \ref{hyp:general1}, then the two inequalities are trivial since $E(\rho) \geq 0$ for all $\rho \in \P_2(\R)$ and $\norm{H\circ \rho}_{L^1(\R)} = E(\rho)$ for all $\rho \in \Pac(\R)$.

Suppose now that $H$ is the density of internal energy for the heat equation. We first prove \eqref{eq:energy-moment} in a way inspired by \cite[Lemma 4.1]{Erbar} and \cite[Section 4]{JKO}. If $\rho \not\in \Pac(\R)$, then the result is trivial since $E(\rho) = +\infty$ by definition. Let $\rho \in \Pac(\R)$ and split the density of internal energy as
\bes
	H(\rho(x)) = H_+(\rho(x)) + H_-(\rho(x)),
\ees
where the subscripts $+$ and $-$ denote respectively the positive and negative parts; here we choose to define the negative part as being negative. Write $I_\delta := \{x \in \R \mid \rho(x) \leq \exp(-\delta |x|^2)\}$ and $J_\delta := \{x \in \R \mid \exp(-\delta |x|^2) < \rho(x) \leq 1\}$ and recall that $x|\log x| \leq \sqrt{x}$ for all $x\in[0,1]$ and that $x\mapsto |\log x|$ is decreasing on $(0,1]$. Compute
\begin{align*}
	-\int_{\R} H_-(\rho(x)) \d x &= \int_{\{y \in \R \mid \rho(y) \leq 1\}} \rho(x) |\log\rho(x)| \d x= \int_{I_\delta} \rho(x) |\log\rho(x)| \d x + \int_{J_\delta} \rho(x) |\log\rho(x)| \d x\\
	&\leq \int_{I_\delta} \sqrt{\rho(x)} \d x + \int_{J_\delta} \rho(x) |\log (\exp(-\delta |x|^2))| \d x\\
	&\leq \int_{\R} e^{-\delta|x|^2/2} \d x + \delta \int_{\R} |x|^2 \rho(x) \d x = K_\delta + \delta M_2(\rho),
\end{align*}
which shows \eqref{eq:energy-moment} as $H_+(\rho(x)) \geq 0$ for all $x\in\R$. To prove \eqref{eq:energy-l1} use the computation above to get
\begin{align*}
	\norm{H\circ \rho}_{L^1(\R)} &= \int_{\{y \in \R \mid \rho(y) \leq 1\}} \rho(x) |\log\rho(x)| \d x + \int_{\{y \in \R \mid \rho(y) > 1\}} \rho(x) \log\rho(x) \d x\\
	&\leq K_\delta + \delta M_2(\rho) + E(\rho) + \int_{\{y \in \R \mid \rho(y) \leq 1\}} \rho(x) |\log\rho(x)| \d x \leq 2K_\delta + 2\delta M_2(\rho) + E(\rho),
\end{align*}
which is \eqref{eq:energy-l1}.
\end{proof}

\begin{lem}\label{lem:bounds-moment-energy}
	Take $\O = \R$. Let $H$ be as in \eqref{hyp:heat} or let it satisfy \ref{hyp:general1}, and let $(\mu_N)_{N\geq2}$ be as in Theorem \ref{thm:gradflow}. Then there exist two finite constants $M_0(T) > 0$ and $E_0(T) \in \R$ such that the following bounds hold for all $t \in [0,T]$.
	\be\label{eq:moment-discrete}
		M_2(\mu_N(t)) \leq M_0(T)
	\ee
and
	\be \label{eq:energy-bound-discrete}
		\E(\mu_N(t)) \geq E_0(T).
	\ee
\end{lem}
\begin{proof}
	Fix $t\in [0,T]$ and consider $\delta>0$ that we choose later. Consider $\rho_N(t)$, the piecewise constant density defined in \eqref{eq:density-on-balls} associated to $\mu_N(t)$. By \eqref{eq:energy-moment},
\be\label{eq:inequality-discrete}
	\E(\mu_N(t)) = E(\rho_N(t)) \geq -K_\delta -\delta M_2(\rho_N(t)).
\ee
Let us compute
\bes
	M_2(\rho_N(t)) = \frac1N \sum_{i=1}^N \frac{1}{r_i(t)} \int_{x_i(t)-\frac{r_i(t)}{2}}^{x_i(t)+\frac{r_i(t)}{2}} x^2 \d x = \frac1N \sum_{i=1}^N \left(x_i(t)^2 + \frac{r_i(t)^2}{12}\right) = M_2(\mu_N(t)) + \frac1N\sum_{i=1}^N\frac{r_i(t)^2}{12},
\ees
where we recall that $r_i(t) = \min(\Delta x_{i}(t), \Delta x_{i+1}(t))$. Let us write $k:= \argmin\{i\in\{1,\dots,N\} \mid x_i(t) \geq 0\}$. Then $r_i(t) \leq \Delta x_i(t) \leq x_i(t)$ for all $i \geq k+1$, $r_i(t) \leq \Delta x_{i+1}(t) \leq -x_i(t)$ for all $i \leq k-2$, $r_k(t) \leq \Delta x_{k+1}(t) \leq x_{k+1}(t)$ and $r_{k-1}(t) \leq \Delta x_{k-1}(t)\leq -x_{k-2}(t)$. Therefore,
\bes
	\sum_{i=1}^N r_i^2 = \sum_{\substack{i=1\\i\neq k-1,k}}^N r_i^2 + r_k^2 + r_{k-1}^2\leq \sum_{\substack{i=1\\i\neq k-1,k}}^N x_i^2 + x_{k+1}^2 + x_{k-2}^2 \leq 2NM_2(\mu_N),
\ees
where we omitted the time dependences. Thus,
\bes
	M_2(\rho_N(t)) \leq M_2(\mu_N(t)) + M_2(\mu_N(t))/6 = 7M_2(\mu_N(t))/6 .
\ees
Then, by \eqref{eq:inequality-discrete},
\be\label{eq:inequality-discrete2}
	\textstyle{\E(\mu_N(t)) \geq -K_\delta - \frac{7\delta}{6}M_2(\mu_N(t)).}
\ee
Now, by the evolution variational inequality given in \cite[Theorem 5.3(iii)]{Ambrosio} and the convexity of $\E$, and since $\E$ is a Lyapunov functional for the discrete gradient flow, we know that, for all $0\leq \tau \leq t \leq T$,
\begin{align*}
	d_2^2(\mu_N(t),\mu_N(\tau)) &\leq 2\int_\tau^t ( \E(\mu_N(\tau)) - \E(\mu_N(s)) ) \d s\\
	&\leq 2\int_\tau^t ( \E(\mu_N^0) - \E(\mu_N(t))) \d s = K_N(t) (t-\tau),
\end{align*}
where $K_N(t) :=2( \E(\mu_N^0) - \E(\mu_N(t)))$. By only swapping $t$ and $\tau$ when $0\leq t < \tau \leq T$, we get
\be \label{eq:holder}
	d_2^2(\mu_N(t),\mu_N(\tau)) \leq K_N(t) |t-\tau| \quad \mbox{for all $t,\tau\in[0,T]$}.
\ee
By Remark \ref{rem:g-finite} and the well-preparedness of $\mu_N^0$ we know that there exists a constant $e_0 = e(\rho_0) \in \R$ such that $\E(\mu_N^0)\leq e_0$, and therefore, by \eqref{eq:inequality-discrete2},
\be\label{eq:constant-ineq}
	\textstyle{K_N(t) \leq 2e_0 + 2K_\delta +\frac{7\delta}{3}M_2(\mu_N(t)).}
\ee
Note that
\begin{align*}
	M_2(\mu_N(t)) = d_2^2(\mu_N(t),\delta_0) &\leq \left(d_2^2(\mu_N(t),\mu_N^0) + d_2(\mu_N^0,\delta_0)\right)^2\\
	&\leq 2d_2^2(\mu_N(t),\mu_N^0) + 2d_2^2(\mu_N^0,\delta_0) = 2d_2^2(\mu_N(t),\mu_N^0) + 2M_2(\mu_N^0).
\end{align*}
By assumption $\rho_0\in\G(\R)$, which, together with the well-preparedness of $\mu_N^0$, implies the existence of $m_0= m(\rho_0)>0$ such that $M_2(\mu_N^0) \leq m_0$. This, along with \eqref{eq:holder} and \eqref{eq:constant-ineq}, implies
\bes
	\textstyle{M_2(\mu_N(t)) \leq 2\left( 2e_0 + 2K_\delta +\frac{7\delta}{3}M_2(\mu_N(t)) \right) t + 2m_0 \leq 4\left( e_0 + K_\delta +\frac{7\delta}{6}M_2(\mu_N(t)) \right) T + 2m_0.}
\ees
Hence, by choosing any $\delta< 3/(14T)$, say $\delta = 3/(28T)$, we get
\bes
	M_2(\mu_N(t)) \leq 8( e_0 + K_{3/(28T)} ) T + 4m_0 =: M_0(T),
\ees
which is \eqref{eq:moment-discrete}. To prove \eqref{eq:energy-bound-discrete}, use \eqref{eq:inequality-discrete2} to get
\bes
	\textstyle{\E(\mu_N(t)) \geq -K_\delta - \frac{7\delta}{6}M_0(T) =: E_0(T),}
\ees
for any choice of $\delta >0$, which ends the proof.
\end{proof}

We can finally get the compactness part of the main theorem and \ref{cond:md}.
\begin{lem}\label{lem:compactness}
	Let $H$ be as in \eqref{hyp:heat} or let it satisfy \ref{hyp:general1}, and let $(\mu_N)_{N\geq2}$ be as in Theorem \ref{thm:gradflow}. Then there exists $\rho\in AC^2([0,T],\P_2(\O))$ such that $\mu_N(t) \wto \rho(t)$ narrowly as $N\to\infty$ for all $t \in [0,T]$. Furthermore, \ref{cond:md} holds.
\end{lem}
\begin{proof}
	We want to use the generalisation of the Arzel\`a--Ascoli theorem given in \cite[Proposition 3.3.1]{AGS}. To this end we first show that the family $\{\mu_N\}_{N\geq2}$ is equi-continuous in time on $[0,T]$ with respect to $d_2$. By Lemma \ref{lem:bounds-moment-energy} and its proof, we already have
\bes
	d_2^2(\mu_N(t),\mu_N(\tau)) \leq 2(e_0 - E_0(T)) |t-\tau| \quad \mbox{for all $t,\tau\in[0,T]$},
\ees
which is the result---more specifically, this shows that $\mu_N$ is $(1/2)$-H\"older continuous, uniformly in $N$. Now, fix $t \in [0,T]$. In order to apply \cite[Proposition 3.3.1]{AGS}, we now only need to show that the family $\{\mu_N(t)\}_{N\geq2}$ is narrowly sequentially compact; by Prohorov's theorem, this means showing that $\{\mu_N(t)\}_{N\geq2}$ is tight, uniformly in $t$, which in turn is implied by $(M_2(\mu_N(t)))_{N\geq2}$ being a sequence bounded uniformly in $N$ and $t$, which is readily given to us by \eqref{eq:moment-discrete}. The Arzel\`a--Ascoli theorem then gives that there exists $\rho \in C([0,T],\P(\O))$, a continuous curve from $[0,T]$ to $\P(\O)$, such that $\mu_N(t) \wto \rho(t)$ narrowly as $N\to\infty$ for all $t\in[0,T]$, up to a subsequence of $(\mu_N(t))_{N\geq2}$; we also have $\rho(t) \in \P_2(\O)$ for all $t\in[0,T]$. 

We now show that $\rho$ is actually in $AC^2([0,T],\P_2(\O))$ and that \ref{cond:md} is true. This part of the proof is based on \cite[Theorem 5.6]{CT}. Fix $t \in [0,T]$. By Remark \ref{rem:prop-gradflow}, 
\bes
	\int_0^t |\mu_N'|_{d_2}(s)^2 \d s = \E(\mu_N^0) - \E(\mu_N(t)) \leq e_0 - E_0(T),
\ees
where $e_0$ and $E_0(T)$ are as above. Then, up to a subsequence, $\lim_{N\to\infty} \int_0^t |\mu_N'|_{d_2}(s)^2 \d s = C$ for some $C\geq0$ independent of $N$. Therefore, $|\mu_N'|_{d_2}$ is bounded in $L^2([0,t])$, and so, up to a further subsequence, it is $L^2$-weakly convergent to some $v \in L^2([0,t])$. It is then also $L^1$-weakly convergent to $v$, so that
\be\label{eq:L1}
	\lim_{N\to\infty} \int_{t_0}^{t_1} |\mu_N'|_{d_2}(s) \d s = \int_{t_0}^{t_1} v(s) \d s \quad \mbox{for all $0\leq t_0\leq t_1 \leq T$}.
\ee
We also know that, by definition of the metric derivative and $\mu_N$ being $2$-absolutely continuous, 
\bes
	d_2(\mu_N(t_0),\mu_N(t_1)) \leq \int_{t_0}^{t_1} |\mu_N'|_{d_2}(s) \d s.
\ees
Then, by the narrow lower semi-continuity of $d_2$, see \cite[Proposition 3.5]{AG}, and \eqref{eq:L1}, 
\bes
	d_2(\rho(t_0),\rho(t_1)) \leq \int_{t_0}^{t_1} v(s) \d s.
\ees
Therefore $\rho \in AC^2([0,T],\P_2(\O))$ and, by the remark below Definition \ref{defn:absolutely-continuous}, $|\rho'|_{d_2}(s) \leq v(s)$ for almost every $s \in [0,T]$. By the weak lower semi-continuity of the $L^2$-norm, this gives
\bes
	\liminf_{N\to\infty} \int_0^t |\mu_N'|_{d_2}(s)^2 \d s = \lim_{N\to\infty} \int_0^t |\mu_N'|_{d_2}(s)^2 \d s \geq \int_0^t v(s)^2 \d s \geq \int_0^t |\rho'|_{d_2}(s)^2 \d s,
\ees
which is \ref{cond:md}.
\end{proof}

\section{Condition on the energy and $\Gamma$-convergence of the discrete energy}
\label{sec:gamma-convergence}

In this section we prove that \ref{cond:liminf} holds. We also prove that the discrete energy given in \eqref{eq:energy-discrete} actually $\Gamma$-converges with respect to $d_2$, in dimension one, to the continuum energy functional \eqref{eq:energy} as the number of particles $N$ grows to infinity; this justifies the existence of a well-prepared sequence for $\rho_0\in\G(\O)$ assumed in Theorem \ref{thm:gradflow}.

\subsection{Condition on the energy}\label{subsec:energy}
We directly give the proof of \ref{cond:liminf}. Note that the proof of Lemma \ref{lem:narrow-rhoN} says that the piecewise constant density $\rho_N$, defined in \eqref{eq:density-on-balls}, is a good narrow approximation of the limiting $\rho$ of Theorem \ref{thm:gradflow}.
\begin{lem}\label{lem:narrow-rhoN}
	Let $H$ be as in \eqref{hyp:heat} or let it satisfy \ref{hyp:general1}, and let $(\mu_N)_{N\geq2}$ and $\rho$ be as in Theorem \ref{thm:gradflow}. Then \ref{cond:liminf} holds.
\end{lem}
\begin{proof}
	Let us omit the time dependences. Write $\phi_N:= \left| \int_{\R} \varphi(x) \rho_N(x) \d x - \int_{\R} \varphi(x) \d\mu_N(x) \right|$ for some $\varphi \in C_\mt{b}(\R)$ Lipschitz with constant $L>0$. Notice that we integrate over $\R$ even if $\O=[-\ell,\ell]$ since $\rho_N \in \Pac(\R)$. Compute
\begin{align*}
	\phi_N &= \left| \sum_{i=1}^N \frac1N \int_{B_i} \frac{\varphi(x)}{|B_i|} \d x - \frac1N \sum_{i=1}^N \varphi(x_i) \right| \leq \frac1N \sum_{i=1}^N \frac{1}{|B_i|} \int_{B_i} | \varphi(x) - \varphi(x_i)| \d x\\
	& \leq \frac LN \sum_{i=1}^N \frac{1}{|B_i|} \int_{B_i} | x - x_i| \d x \leq \frac LN \sum_{i=1}^N |B_i| \leq \frac LN \sum_{i=2}^N \Delta x_i + \frac{L\Delta x_2}{N}\\
	& \leq \frac{L}{N}\left(\sqrt{2(x_1^2 + x_N^2)} + \sqrt{2(x_1^2 + x_2^2)}\right) \leq 2L\sqrt{\frac{2M_2(\mu_N)}{N}} \leq 2L\sqrt{\frac{2M_0(T)}{N}} \xrightarrow[N\to\infty]{}0,
\end{align*}
by \eqref{eq:moment-discrete}, which shows that $\rho_N-\mu_N\wto0$ narrowly as $N\to\infty$. Then, since $\mu_N \wto \rho$, we get that $\rho_N \to \rho$ as $N\to\infty$. Now, by definition, $\E(\mu_N) = E(\rho_N)$, and $E$ is narrowly lower semi-continuous, which gives \ref{cond:liminf}.
\end{proof}

\subsection{$\Gamma$-convergence of the discrete energy}\label{subsec:gamma}
We show that the discrete energy $\Gamma$-converges to the continuum one with respect to the metric $d_2$, see Definition \ref{defn:gamma-convergence}. We do not show it with respect to the narrow convergence if $\O = \R$ (for which $d_2$- and narrow convergences are not equivalent); indeed, this case is more involved since the ``liminf" condition may not hold for sequences which do not have a control on the second moments, see the proof of Lemma \ref{lem:narrow-rhoN}.

\begin{defn}[$\Gamma$-convergence] \label{defn:gamma-convergence}
	We say that the discrete energy $(\E)_{N\geq2}$ \emph{$\Gamma$-converges} (with respect to $d_2$) to the continuum energy $E$ if the following two conditions are met for all $\rho \in \P_2(\O)$.
\begin{enumerate}[label=(\roman*)]
	\item \label{it:liminf} (``liminf" condition) All sequences $(\mu_N)_{N\geq2}$ with $\mu_N\in\A_{N,w}(\O)$ for all $N\geq2$ such that $d_2(\mu_N,\rho) \to 0$ as $N \to \infty$ satisfy $E(\rho) \leq \liminf_{N \to \infty} \E(\mu_N)$.
	\item \label{it:limsup} (``limsup" condition) There exists a recovery sequence with respect to $d_2$ for $\rho$.
\end{enumerate}
\end{defn}

On top of \ref{hyp:general1}, in this subsection we sometimes assume that $H$ satisfies the following: there exist continuous functions $f_1,f_2 \: [0,\infty) \to \R$ such that $f_1(1) = 1$ and $f_2(1) = 0$, and
\begin{equation} \label{eq:upper-bound}
	H(\alpha x) \leq f_1(\alpha)H(x) + f_2(\alpha)x \quad \mbox{for all $x,\alpha \in [0,\infty)$}.
\end{equation}
This is still satisfied by typical densities of internal energy, such as for the heat equation and the porous medium equation. This assumption is actually only needed in the proof of Lemma \ref{lem:mollifier}.

\begin{thm}\label{thm:gamma-convergence}
	Let $H$ satisfy \ref{hyp:general1} and \eqref{eq:upper-bound}, or let it be as in \eqref{hyp:heat}. Then $(\E)_{N\geq 2}$ $\Gamma$-converges to $E$.
\end{thm}
Showing the ``liminf" condition follows the same strategy used to prove \ref{cond:liminf} in Lemma \ref{lem:narrow-rhoN}. The difference lies in the fact that convergence in $d_2$ to an element of $\P_2(\O)$ yields by itself a uniform bound on the second moments of the sequence considered, so that \eqref{eq:moment-discrete} is readily given. We therefore only need to prove the ``limsup" condition.

To this end we only need, for any $\rho \in \P_2(\O)$, to find a recovery sequence with respect to $d_2$. Suppose $E(\rho) < +\infty$, or the result is trivial. Then, by definition of the continuum energy, $\rho \in \Pac(\O)$. We proceed in two main stages: we first prove the result for any $\rho \in \G(\O)$, and then relax this assumption on $\rho$ and prove the general result for any $\rho \in \Pac(\O)$ by a density argument. 

In order to be able to apply Theorem \ref{thm:serfaty} we actually do not need to show the $\Gamma$-convergence on the whole set $\P_2(\O)$, but rather only on the smooth set $\G(\O)$. Indeed, we only need to find a recovery sequence (which is also well-prepared) for the initial profile $\rho_0 \in \G(\O)$, which we do in Section \ref{subsec:case-g}. Note that in this case the hypothesis \eqref{eq:upper-bound} is not needed, which is why it is not assumed in the main theorem.

\subsubsection{Smooth case}
\label{subsec:case-g}

Let us introduce the notion of pseudo-inverse.

\begin{defn}[Pseudo-inverse] \label{defn:generalised-inverse}
	Let $F\: \O \to [0,1]$ be a non-decreasing and right-continuous function. The \emph{pseudo-inverse} of $F$ is the non-decreasing and right-continuous function defined by $\Phi \: [0,1] \to \O \cup \{-\infty,+\infty\}$ and
\bes
	\begin{cases} \Phi(\eta) = \inf \{x \in \O\mid F(x) > \eta\} & \mbox{for all $\eta \in [0,1)$},\\ \Phi(1) = \lim_{\eta \to 1^-} \Phi(\eta). \end{cases}
\ees
\end{defn}

	If $F$ is the cumulative distribution function of a probability density $\rho$, then $\Phi\in L^2([0,1])$ if and only if $\rho \in \P_2(\O)$. If $\rho \in \G(\O)$, then $\Phi \in C^2([0,1])$ is increasing and is the classical inverse of $F$. 

A recovery sequence for any $\rho\in\G(\O)$ is given in the following lemma.

\begin{lem}\label{lem:limsup-continuous}
	Let $H$ satisfy \ref{hyp:general1} or let it be as in \eqref{hyp:heat}. Let $\rho \in \G(\O)$, $F\: \O \to [0,1]$ be its cumulative distribution function, and $\Phi$ be the pseudo-inverse of $F$. Then the sequence $(\mu_N)_{N\geq2}$ with $\mu_N\in\A_{N,w}(\O)$ for all $N\geq2$ and particles
\be\label{eq:initial-set}
	\begin{cases} x_1 = \Phi(0),\\ x_i = \Phi\left(\frac iN \right) & \mbox{for $i \in \{2,\dots,N\}$} \end{cases}
\ee
is well-prepared for $\rho$ according to Definition \ref{defn:initial-set}. 
\end{lem}
\begin{proof}
	We first show the bound condition on the inter-particle distances. Notice that
\bes
	\Delta x_2 = \Phi\left(\frac2N\right) - \Phi(0), \quad \Delta x_i = \Phi\left(\frac iN\right) - \Phi\left(\frac{i-1}{N}\right) \quad \mbox{for $i\in\{3,\dots,N\}$.}
\ees
Since $\Phi\in C^2([0,1])$, the mean-value theorem yields
\bes
	\Delta x_2 = (\Phi'(\xi_1) + \Phi'(\xi_2))/N = (N\rho(\Phi(\xi_1)))^{-1} + (N\rho(\Phi(\xi_2)))^{-1}, \quad \Delta x_i = \Phi'(\xi_i)/N = (N\rho(\Phi(\xi_i)))^{-1}
\ees
for $i\in\{3,\dots,N\}$ for some $\xi_i \in ((i-1)/N,i/N)$. Therefore,
\be\label{eq:bounds-distances}
	\textstyle (N\max\rho)^{-1} \leq \Delta x_i \leq 2(N\min_{\supp\rho}\rho)^{-1} \quad \mbox{for all $i\in\{2,\dots,N\}$}.
\ee
Also, one sees that $x_N = -x_1 = r$, where $r$ is as in Definition \ref{defn:smooth}, as required by Definition \ref{defn:initial-set}.

We now show that $(\mu_N)_{N\geq2}$ is a recovery sequence with respect to $d_2$ for $\rho$. First, let us show that $d_2(\mu_N,\rho) \to 0$ as $N \to \infty$. We know that the quadratic Wasserstein distance can be written in one dimension as
\bes
	d_2^2(\mu_N,\rho) = \int_0^1 \left(\Gamma_N(\eta) - \Phi(\eta)\right)^2 \d \eta,
\ees
where $\Gamma_N$ is the pseudo-inverse of the cumulative distribution function of $\mu_N$. Also,
\bes
	\Gamma_N(\eta) = \begin{cases} x_1 = \Phi(0) & \mbox{if $\eta\in\left[0,\frac 1N\right)$},\\ x_i = \Phi\left(\frac iN\right) & \mbox{for all $i \in \{2,\dots,N\}$, if $\eta\in\left[\frac{i-1}{N},\frac iN\right)$}.\end{cases}
\ees
Hence, writing $\Delta_i \Phi := \Phi(i/N) - \Phi((i-1)/N)$ for all $i\in\{1,\dots,N\}$,
\begin{align*}
	d_2^2(\mu_N,\rho) &= \sum_{i=2}^N \int_{\frac{i-1}{N}}^{\frac{i}{N}}  \left(\Phi\left(\frac iN \right) - \Phi(\eta)\right)^2 \d \eta + \int_0^{\frac{1}{N}}  \left(\Phi(0) - \Phi(\eta)\right)^2 \d \eta\\
	&\leq \sum_{i=1}^N \int_{\frac{i-1}{N}}^{\frac{i}{N}}  \left(\Delta_i \Phi\right)^2 \d \eta = \dfrac 1N \sum_{i=1}^N \left(\Delta_i \Phi\right)^2
\end{align*}
since $\Phi$ is increasing. Notice that $\Delta x_2 = \Delta_1\Phi + \Delta_2\Phi$ and $\Delta x_i = \Delta_i\Phi$ for all $i\in\{3,\dots,N\}$, and so
\bes
	\sum_{i=1}^N \left(\Delta_i \Phi\right)^2 = \sum_{i=2}^N \Delta x_i^2 - 2\Delta_1\Phi \Delta_2\Phi \leq \sum_{i=2}^N \Delta x_i^2.
\ees
By \eqref{eq:bounds-distances}, we know that $\Delta x_i \leq 2/(N\min_{\supp\rho}\rho)$, which then gives
\bes
	d_2^2(\mu_N,\rho) \leq \frac{1}{N} \sum_{i=2}^N \Delta x_i^2 \leq \textstyle 4(N\min_{\supp\rho}\rho)^{-2} \xrightarrow[N\to\infty]{}0.
\ees
Let us now prove that $\E(\mu_N) \to E(\rho)$ as $N \to \infty$. In the rest of the proof, $K\in\R$ denotes a generic constant which only depends on $\rho$ and $H$ and which may take different values throughout computations. Since $x_1$ and $x_N$ are at the boundaries of the support of $\rho$, compute
\begin{align*}
	E(\rho) &= \sum_{i=2}^N \int_{x_{i-1}}^{x_i} H(\rho(x)) \d x = \sum_{i=2}^N \Delta x_i H(\rho(x_i)) + \sum_{i=2}^N \int_{x_{i-1}}^{x_i} \rho'(\xi_i(x)) H'(\rho(\xi_i(x))) (x-x_i)\d x\\
	&\geq \sum_{i=2}^N \Delta x_i H(\rho(x_i)) + K \sum_{i=2}^N \Delta x_i^2 \geq \sum_{i=2}^N \Delta x_i H(\rho(x_i)) + \frac{K}{N \min_{\supp\rho}\rho^2},
\end{align*}
where $\xi_i\: [x_{i-1},x_i] \to (x_{i-1},x_i)$ for all $i\in\{2,\dots,N\}$ are continuous and bounded functions coming from the mean-value theorem, and $K<0$ is indeed $i$- and $N$-independent since $H\in C^\infty((0,\infty))$ and $\rho\in\G(\O)$ so that $\rho'$ is bounded and $\rho$ is bounded away from 0. Using a second order Taylor expansion on $\Phi$, and again the boundedness properties of $\rho$ and $\rho'$,
\bes
	\Delta x_i = (N\rho(x_i))^{-1} + KN^{-2} \quad \mbox{for all $i\in\{3,\dots,N\}$}.
\ees
Then one has
\bes
	\frac{\Delta x_{i+1}}{\Delta x_i} = \frac{\rho(x_i)}{\rho(x_{i+1})} + \frac{K}{N} \quad \mbox{for all $i\in\{3,\dots,N\}$}.
\ees
Therefore, $r_i := \Delta x_i \min(1,\Delta x_{i+1}/\Delta x_i) \sim \Delta x_i + K\Delta x_i/N$ as $N\to\infty$, uniformly in $i$. Thus,
\begin{align*}
	\sum_{i=2}^N \Delta x_i H(\rho(x_i)) = \Delta x_2H(\rho(x_2)) + \sum_{i=3}^N r_i H\left(\frac{1}{Nr_i}\right) + \frac{K}{N} \geq \sum_{i=3}^N r_i H\left(\frac{1}{Nr_i}\right) + \frac{K}{N},
\end{align*}
where in the last inequality the term $\Delta x_2 H(\rho(x_2))$ is absorbed in the term $K/N$. All in all, we get
\be\label{eq:sup-continuum}
	E(\rho) \geq \sum_{i=3}^N r_i H\left(\frac{1}{Nr_i}\right) + \frac{K}{N} = \frac{1}{N} \sum_{i=3}^N h(Nr_i) + \frac{K}{N}.
\ee
Then, by \eqref{eq:energy-discrete} and \eqref{eq:sup-continuum},
\bes
	\E(\mu_N) \leq \frac1N\sum_{i=3}^N h(Nr_i) + \frac{K}{N} \leq E(\rho) + \frac{K}{N},
\ees
which leads to the result by taking $\limsup$ as $N\to\infty$.
\end{proof}

\subsubsection{General case}
We now relax our assumptions on $\rho$ and consider $\rho \in \Pac(\O)$.

\begin{lem} \label{lem:mollifier}
	Let $H$ be as assumed in Theorem \ref{thm:gamma-convergence}. Let $\rho \in \Pac(\O)$, and define $G_\delta(x) := (1/\delta \sqrt{2\pi}) \exp(-x^2/(2\delta^2))$ for all $x\in\R$ and $\delta>0$. Write
\bes
	\rho_\delta := \frac{\chi_{S(\O) \cap B_{1/\delta}(0)}G_\delta \ast \rho}{\|G_\delta \ast \rho\|_{L^1(S(\O) \cap B_{1/\delta}(0))}} \quad \mbox{with} \quad S(\O) = \begin{cases} \supp \rho & \mbox{if $\O = \R$},\\ [-\ell,\ell] & \mbox{if $\O = [-\ell,\ell]$}. \end{cases}
\ees
Then $\rho_\delta \in \G(\O)$ for all $\delta > 0$, $d_2(\rho_\delta,\rho) \to 0$ as $\delta \to 0$, and $\limsup_{\delta \to 0} E(\rho_\delta) \leq E(\rho)$.
\end{lem}
\begin{proof}
	Write $S_\delta(\O) := S(\O) \cap B_{1/\delta}(0)$ and $g_{\delta}(\O):= \|G_\delta \ast \rho\|_{L^1(S_\delta(\O))}$ for all $\delta > 0$. Checking that $\rho_\delta \in \G(\O)$ for all $\delta >0$ is straightforward, whereas $d_2(\rho_\delta,\rho) \to 0$ comes from the facts that $g_{\delta}(\O) \to 1$ as $\delta \to0$ and that convolutions of probability measures with finite second moments converge in $d_2$ to their original measures, see \cite[Lemma 7.1.10]{AGS}. Then
\bes
	E(\rho_\delta) = \int_{\O} H(\rho_\delta(x)) \d x = \int_{S_\delta(\O)} H\left(\frac{G_\delta \ast \rho(x)}{g_\delta(\O)} \right) \d x.
\ees
By Jensen's inequality and Fubini's theorem, we deduce
\bes
	E(\rho_\delta) \leq \int_{S_\delta(\O)} G_\delta \ast H\left(\frac{\rho(x)}{g_\delta(\O)}\right) \d x = \int_{\O} H\left(\frac{\rho(y)}{g_\delta(\O)}\right) G_\delta\ast\chi_{S_\delta(\O)}(y) \d y.
\ees
By \eqref{eq:upper-bound} we get
\begin{align*}
	E(\rho_\delta) &\leq f_1\left(\frac{1}{g_\delta(\O)}\right) \int_{\O} H(\rho(y)) G_\delta\ast\chi_{S_\delta(\O)}(y) \d y + f_2\left(\frac{1}{g_\delta(\O)}\right) \int_{\O} \rho(y) G_\delta\ast\chi_{S_\delta(\O)}(y) \d y.
\end{align*}
Now we want to use Lebesgue's dominated convergence theorem. First note that $G_\delta\ast\chi_{S_\delta(\O)} \leq 1$ for all $\delta >0$ and $G_\delta\ast\chi_{S_\delta(\O)} \to 1$ as $\delta \to0$ pointwise. Since $E(\rho) = \int_\O H(\rho(y)) \d y$ is assumed to be finite, we know by \eqref{eq:energy-l1} that $H\circ\rho \in L^1(\O)$; also, recall that $\int_\O \rho = 1$. We can therefore pass to the limit $\delta\to0$ inside the integrals of the inequality above. Then, by the assumptions on the functions $f_1$ and $f_2$, and since $g_\delta(\O) \to1$ as $\delta \to0$, we get the desired result. 
\end{proof}

We can finish the proof of Theorem \ref{thm:gamma-convergence}.
\begin{proof}[Proof of Theorem \ref{thm:gamma-convergence}]
	Let $\rho\in\P_2(\O)$. As already mentioned, we only need to find a recovery sequence for $\rho$. For all $\delta >0$, let $\rho_\delta$ be as defined in Lemma \ref{lem:mollifier}. By Lemma \ref{lem:limsup-continuous}, the sequence $(\mu_N^\delta)_{N\geq2}$ with particles given in \eqref{eq:initial-set} is a recovery sequence for $\rho_\delta$, i.e., $d_2(\mu_N^\delta,\rho_\delta) \to0$ as $N\to\infty$ and $\limsup_{N\to\infty} E(\mu_N^\delta) \leq E(\rho_\delta)$. Therefore, for every $\delta>0$, there exists $N(\delta)\geq2$ such that
\bes
	E(\mu_{N(\delta)}^\delta) \leq E(\rho_\delta) + \delta \quad \mbox{and} \quad d_2(\mu_{N(\delta)}^\delta,\rho) \leq \delta+d_2(\rho_\delta,\rho),
\ees
where the second inequality is obtained as in Remark \ref{rem:stability-initial}. By Lemma \ref{lem:mollifier} we also have $d_2(\rho_\delta,\rho)\to0$ as $\delta \to0$ and $\limsup_{\delta \to 0} E(\rho_\delta) \leq E(\rho)$, which gives
\bes
	\limsup_{\delta \to 0} E(\mu_{N(\delta)}^\delta) \leq \limsup_{\delta\to0}(E(\rho_\delta) + \delta) \leq E(\rho) \quad \mbox{and} \quad \lim_{\delta \to0} d_2(\mu_{N(\delta)}^\delta,\rho) \leq \lim_{\delta \to0} (\delta+ d_2(\rho_\delta,\rho)) = 0.
\ees
The subsequence $(\mu_{N(\delta)}^\delta)_{\delta>0}$ is therefore a recovery sequence for $\rho \in \Pac(\O)$.
\end{proof}

\section{Condition on the local slopes}
\label{sec:conditions}

To finish the proof of Theorem \ref{thm:gradflow}, we need to check \ref{cond:slopes}. In this section, we write $g:=|\p E|$ and $g_N:=|\p E_N|$, and we take $\O = [-\ell,\ell]$. In this case, by the doubling condition \eqref{eq:doubling}, the local slope $g$ of $E$ is given in the lemma below, see \cite[Theorem 4.16]{Ambrosio}.
\begin{lem} \label{lem:sugE}
Let $H$ be as in \eqref{hyp:heat} or let it satisfy \ref{hyp:general1}. The local slope of $E$ is given by
\bes
	g(\rho) = \sqrt{I(\rho)} \quad \mbox{for all $\rho \in D(I)$},
\ees
where the \emph{(generalised) Fisher information} $I\: \P_2([-\ell,\ell]) \to [0,\infty]$ is defined by
\be\label{eq:fisher}
	I(\rho) = \begin{cases} 
		\displaystyle \int_{-\ell}^\ell \rho'(x)^2H''(\rho(x))^2\rho(x) \d x & \begin{array}{l} \mbox{if $\rho\in\Pac([-\ell,\ell])$}\\ \mbox{and $(H'\circ \rho)\rho -H\circ \rho \in W^{1,1}([-\ell,\ell])$,} \end{array}\\
		+\infty & \begin{array}{l} \mbox{otherwise.} \end{array}
	\end{cases}
\ee
\end{lem}

We want to prove the following, i.e., the condition \ref{cond:slopes}.

\begin{lem} \label{lem:gradflow-cond2}
	Let $H$ be as in \eqref{hyp:heat}, or let it satisfy \ref{hyp:general1} and \ref{hyp:general2}, and let $(\mu_N)_{N\geq2}$ and $\rho$ be as in Theorem \ref{thm:gradflow}. Then we have
\bes
	\displaystyle \liminf_{N\to \infty} g_N(\mu_N(t)) \geq g(\rho(t)) \quad \mbox{for all $t \in [0,T]$}.
\ees
\end{lem}

We proceed progressively for the sake of readability: we first place ourselves in the case of the heat equation and then extend the result to the general case when $H$ satisfies \ref{hyp:general1} and \ref{hyp:general2}.

\subsection{The heat equation}\label{subsec:he}

\subsubsection{Preliminaries}
We want to compute explicitly the local slope $g_N$ of $\E$. First, note that this is identifiable with the minimal norm element of the subdifferential of $\wtE$. Indeed,
\be \label{eq:sugEN}
	g_N(\mu_N) = |\p_w^0\wtE(\bxN)|_w \quad \mbox{for all $\mu_N \in \A_{N,w}([-\ell,\ell])$ with particles $\bxN \in [-\ell,\ell]_w^N$}.
\ee
This comes from \cite[Proposition 1.4.4]{AGS} and the Hilbert structure of $[-\ell,\ell]_w^N$, see Definition \ref{defn:weighted-ip}. Then we need to compute the subdifferential of $\wtE$ and its minimal norm element. To this end, introduce the following notation. Given $\bxN \in[-\ell,\ell]_w^N$, we write $(\lambda^-,\lambda,\lambda^+)\in \Lambda(\bxN)$ if
\bes
	\lambda_i^- \begin{cases} = 0 & \mbox{if $\Delta x_i>\Delta x_{i-1}$},\\ \in [0,1] & \mbox{if $\Delta x_i=\Delta x_{i-1}$},\\ = 1 & \mbox{if $\Delta x_{i}<\Delta x_{i-1}$},\end{cases} \quad \lambda_i \begin{cases} = 0 & \mbox{if $\Delta x_{i+1}>\Delta x_i$},\\ \in [0,1] & \mbox{if $\Delta x_{i+1}=\Delta x_i$},\\ = 1 & \mbox{if $\Delta x_{i+1}<\Delta x_{i}$},\end{cases} \quad \lambda_i^+ \begin{cases} = 0 & \mbox{if $\Delta x_{i+2}>\Delta x_{i+1}$},\\ \in [0,1] & \mbox{if $\Delta x_{i+2}=\Delta x_{i+1}$},\\ =1 & \mbox{if $\Delta x_{i+2}<\Delta x_{i+1}$}\end{cases}
\ees
for all $i\in \{1,\dots,N\}$, with the convention that $\Delta x_1 > \Delta x_0$ and $\Delta x_{N+1} > \Delta x_{N+2}$. Note that the triplet $(\lambda^-,\lambda,\lambda^+)$ contains for each particle the answer to the question ``is the closest neighbour to that particle to the right?'', unless both neighbours are at equal distance. With this notation, we can give the following characterisation of $\p_w \wtE$.
\begin{lem}\label{lem:subdiff-EN}
	Take $\bxN\in[-\ell,\ell]_w^N$. We have
\bes
	\p_w \wtE(\bxN) = \big\{ \bz \in \R_w^N \mid \exists\, (\lambda^-,\lambda,\lambda^+) \in \Lambda(\bxN), z_i = (\lambda_i-\lambda_i^+ +1)\psi_{i+1} - (\lambda_i^- -\lambda_i + 1)\psi_i, 1\leq i \leq N \big\},
\ees
where $\psi_i:= 1/\Delta x_i$ for all $i\in\{1,\dots,N\}$.
\end{lem}
\begin{proof}
	Let $i\in\{1,\dots,N\}$. Compute the subdifferential with respect to the coordinate $x_i$, i.e.,
\bes
	\p_w^i \wtE(\bxN) := \left\{z_i\in \R \,\Big|\, \liminf_{y\to x_i} \frac{\wtE(x_1,\dots,x_{i-1},y,x_{i+1},\dots,x_N) - \wtE(\bxN) - \frac{z_i}{N}(y-x_i)}{|y-x_i|} \geq 0\right\}.
\ees
To this end, first check that
\be\label{eq:ri-sub1}
	\textstyle{\frac1N} \p_w^i r_{i-1} = \begin{cases} \{0\} & \mbox{if $\Delta x_{i-1} < \Delta x_i$},\\ [0,1] & \mbox{if $\Delta x_{i-1} = \Delta x_i$},\\  \{1\} & \mbox{if $\Delta x_{i-1} > \Delta x_i$},\end{cases} \quad \textstyle{\frac1N} \p_w^i r_i = \begin{cases} \{1\} & \mbox{if $\Delta x_i < \Delta x_{i+1}$},\\ [-1,1] & \mbox{if $\Delta x_i = \Delta x_{i+1}$},\\  \{-1\} & \mbox{if $\Delta x_i > \Delta x_{i+1}$},\end{cases}
\ee
and
\be\label{eq:ri-sub2}
	\textstyle{\frac1N} \p_w^i r_{i+1} = \begin{cases} \{-1\} & \mbox{if $\Delta x_{i+1} < \Delta x_{i+2}$},\\ [-1,0] & \mbox{if $\Delta x_{i+1} = \Delta x_{i+2}$},\\  \{0\} & \mbox{if $\Delta x_{i+1} > \Delta x_{i+2}$}.\end{cases}
\ee
Note that, for the specific case of the heat equation, \eqref{eq:energy-discrete} reads as $\wtE(\bxN) = (1/N) \sum_{i=1}^N h(Nr_i) = - \log N -(1/N) \sum_{i=1}^{N} \log r_i$. The function $h$ is a smooth convex function on $(0,\infty)$ so that we can apply the sum rule of subdifferential calculus, see \cite[Section 1.3.4]{Mord} for a detailed account on subdifferential calculus. Therefore, since the particle $x_i$ may only appear in $r_{i-1},r_i$ or $r_{i+1}$,
\be\label{eq:subdiff}
	\p_w^i \wtE(\bxN) = \begin{cases} -\textstyle \frac1N\p_w^i \log r_{i-1} -\frac1N\p_w^i\log r_i -\frac1N\p_w^i\log r_{i+1} & \mbox{if $i \in \{2,\dots,N-1\}$,}\\ -\textstyle \frac1N\p_w^1 \log r_1 -\frac1N \p_w^1 \log r_2 & \mbox{if $i=1$}, \\ -\textstyle \frac1N\p_w^N \log r_{N-1} -\frac1N \p_w^N \log r_N & \mbox{if $i=N$}. \end{cases}
\ee
Since $h$ is non-increasing, $\wtE(\bxN) = (1/N) \sum_{i=1}^N \max[h(N\Delta x_i),h(N\Delta x_{i+1})]$; the function $h$ being smooth, this allows one to apply the chain rule of subdifferential calculus in \eqref{eq:subdiff}. Therefore, by \eqref{eq:ri-sub1} and \eqref{eq:ri-sub2} we get:
\bes
	z_i \in \p_w^i \wtE(\bxN) \iff \exists\, \lambda_i^-,\lambda_i,\lambda_i^+ \in \Lambda(\bxN)\, \text{ with }\, z_i = \frac{\lambda_i-\lambda_i^+ +1}{\Delta x_{i+1}} - \frac{\lambda_i^- -\lambda_i + 1}{\Delta x_i},
\ees
which is the result since the subdifferential of the convex function $\wtE$ at $\bxN$ is
\bes
	\p_w\wtE(\bxN) = \p_w^1\wtE(\bxN) \times \dots \times \p_w^N\wtE(\bxN). \qedhere
\ees
\end{proof}

We introduce the following notation. If $\bxN\in[-\ell,\ell]_w^N$, then for each $i \in \{1,\dots,N\}$ we write $x_i \in (A_{i-1},A_i,A_{i+1})$, where, for any $j \in \{i-1,i,i+1\}$, $A_j =$ ``$R$'' if $\Delta x_j > \Delta x_{j+1}$, $A_j =$``$E$'' if $\Delta x_j = \Delta x_{j+1}$, and $A_j =$ ``$L$'' if $\Delta x_j < \Delta x_{j+1}$. By convention we set $A_0 =$ ``$L$'' and $A_{N+1} =$ ``$R$''. The notation ``$R$" stands for ``Right" (the closest particle to the one considered is the right one), ``E'' stands for ``Equal'', and ``$L$" stands for ``Left".

With this notation we give now the minimal norm element in Lemma \ref{lem:cases}. Its proof is direct by choosing the triplets $(\lambda_i^-,\lambda_i,\lambda_i^+)$ of Lemma \ref{lem:subdiff-EN} in the best possible way so as to minimise the absolute value of $z_i$.
\begin{lem}\label{lem:cases}
	Take $\bxN\in[-\ell,\ell]_w^N$ and write $\bz := (z_1,\dots,z_N) = \p_w^0\wtE(\bxN)$. Then, for each $i\in\{1,\dots,N\}$, according to each case, the component $z_i$ is given by the following, where again we write $\psi_i:=1/\Delta x_i$ for all $i\in\{1,\dots,N\}$.
\begin{longtable}{rcl}
$x_i \in  (R,R,L)$ & $\lambda_i^- = \lambda_i = 1, \lambda_i^+ = 0$ & $z_i = 2\psi_{i+1} - \psi_i$\\
$x_i \in  (E,R,L)$ & $\lambda_i^-\in[0,1],\lambda_i = 1,\lambda_i^+=0$ & $z_i = 2\psi_{i+1} - \psi_i$\\
$x_i \in  (L,R,R)$ & $\lambda_i^-  = 0, \lambda_i = \lambda_i^+ = 1$ & $z_i = \psi_{i+1}$\\
$x_i \in  (L,R,E)$ & $\lambda_i^- = 0,\lambda_i = 1,\lambda_i^+\in[0,1]$ & $z_i = \psi_{i+1}$\\
$x_i \in  (L,R,L)$ & $\lambda_i^- = \lambda_i^+ = 0,\lambda_i = 1$ & $z_i = 2\psi_{i+1}$\\
$x_i \in  (R,L,R)$ & $\lambda_i^- = \lambda_i^+ = 1,\lambda_i = 0$ & $z_i = - 2\psi_i  $\\
$x_i \in  (R,L,E)$ & $\lambda_i^- = 1,\lambda_i = 0,\lambda_i^+\in[0,1]$ & $z_i = \psi_{i+1} - 2\psi_i  $\\
$x_i \in  (R,L,L)$ & $\lambda_i^- = 1,\lambda_i = \lambda_i^+ = 0$ & $z_i = \psi_{i+1} - 2\psi_i  $\\
$x_i \in  (E,L,R)$ & $\lambda_i^- \in [0,1], \lambda_i = 0, \lambda_i^+ = 1$ & $z_i = - \psi_i  $\\
$x_i \in  (L,L,R)$ & $\lambda_i^- = \lambda_i = 0,\lambda_i^+ = 1$ & $z_i = -\psi_i  $\\
$x_i \in  (R,R,R)$ & $\lambda_i^- = \lambda_i = \lambda_i^+ = 1$ & $z_i = \psi_{i+1} - \psi_i$\\
$x_i \in  (R,R,E)$ & $\lambda_i^- = \lambda_i = 1, \lambda_i^+ \in [0,1]$ & $z_i = \psi_{i+1} - \psi_i$\\
$x_i \in  (E,R,R)$ & $\lambda_i^- \in [0,1],\lambda_i = \lambda_i^+ =1$ & $z_i = \psi_{i+1} - \psi_i$\\
$x_i \in  (E,R,E)$ & $\lambda_i^-,\lambda_i^+\in[0,1],\lambda_i = 1$ & $z_i = \psi_{i+1} - \psi_i$\\
$x_i \in  (E,L,E)$ & $\lambda_i^-,\lambda_i^+ \in [0,1], \lambda_i = 0$ & $z_i = \psi_{i+1} - \psi_i  $\\
$x_i \in  (E,L,L)$ & $\lambda_i^- \in[0,1], \lambda_i = \lambda_i^+ = 0$ & $z_i = \psi_{i+1} - \psi_i  $\\
$x_i \in  (L,L,E)$ & $\lambda_i^- = \lambda_i = 0,\lambda_i^+ \in [0,1]$ & $z_i = \psi_{i+1} - \psi_i  $\\
$x_i \in  (L,L,L)$ & $\lambda_i^- = \lambda_i = \lambda_i^+ = 0$ & $z_i = \psi_{i+1} - \psi_i  $
\end{longtable}
\noindent Moreover, if $x_i\in(A,E,B)$ for any $(A,B) \in \{R,E,L\}$, i.e., $\Delta x_{i+1} = \Delta x_i$ ($\lambda_i \in [0,1]$), then $z_i = 0$.
\end{lem}

We give now the lemma ensuring that boundary particles stay at the boundary at all times. 
\begin{lem}\label{lem:boundary-particles-fixed}
	Let $\bxN \in [-\ell,\ell]_w^N$ be as in Theorem \ref{thm:gradflow}. Then $x_N(t) = -x_1(t) = \ell$ for all $t \in [0,T]$.
\end{lem}
\begin{proof}
	Since $\bxN$ is well-prepared initially for $\rho_0 \in \G([-\ell,\ell])$, we have $x_N(0) = -x_1(0) = \ell$. We want to show that this holds for all times after 0. Proposition \ref{prop:discrete-gradient-flow-well-posed} tells us that $\bxN'(t) = -\p_w^0 \E(\bxN(t))$ for almost every $t\in[0,T]$; but we actually have $\der^+\bxN/\der t(t) = -\p_w^0 \E(\bxN(t))$ for all $t \in[0,T]$, where $\der^+/\der t$ stands for the right-derivative in time, see \cite[Theorem 3.1]{Brezis}. Suppose first, by contradiction, that $x_1(\tau) > -\ell$ for some arbitrarily small time $\tau>0$. Then, by \eqref{eq:convention2}, $x_0(\tau) = - \infty$, and therefore, by Lemma \ref{lem:cases}, $\der^+ x_1/\der t(\tau) < 0$ since $x_1(\tau) \in (L,R,A)$ for some $A \in \{R,E,L\}$---and analogously for $x_N$. Suppose now that $x_1(\tau) < -\ell$. Then, by \eqref{eq:convention2}, $\Delta x_1(\tau) < \Delta x_2(\tau)$, and therefore, by Lemma \ref{lem:cases}, $\der^+ x_1/\der t(\tau) > 0$ since $x_1(\tau) \in (L,L,A)$ for some $A \in \{R,E,L\}$---and analogously for $x_N$. Hence $\der^+ x_1/\der t(t) = \der^+ x_N/\der t(t) = 0$ for all $t\in[0,T]$, and we get the result.
\end{proof}

The following two lemmas give a control on how the inter-particle distances behave.

\begin{lem} \label{lem:bounds-distances}
	Let $\bxN \in [-\ell,\ell]_w^N$ be as assumed in Theorem \ref{thm:gradflow}. Then, for all $t\in[0,T]$,
\bes
	\textstyle a_1N^{-1} \leq \Delta x_i(t) \leq a_2N^{-1} \quad \mbox{for all $i\in\{2,\dots,N\}$}.
\ees
The constants $a_1$ and $a_2$ are those of Definition \ref{defn:initial-set} for the well-prepared set $\boldsymbol{x_N^0}$ for $\rho_0$.
\end{lem}
\begin{proof}
	We first show the left-hand side inequality. Take a ``curve'' of indices $i\: [0,T]\to \{2,\dots,N\}$ such that $\Delta x_{i(t)}(t) = \min_{j\in\{2,\dots,N\}} \Delta x_j(t)$ for all $t\in[0,T]$. Note that, for $t,\tau\in [0,T]$, $x_{i(t)}(\tau)$ denotes the position of the particle $x_{i(t)}$, that is, the right-particle of any minimal inter-particle interval at time $t$, at time $\tau$; obviously, if $t\neq\tau$, $\Delta x_{i(t)}(\tau)$ may not be equal to the minimal inter-particle distance at time $\tau$. We have, for all $t\in[0,T]$,
\bes
	x_{i(t)}(t) \in \bigcup_{\substack{(A,B,C)\in\{R,E\}\times\{E,L\}\\ \phantom{{}===={}}\times\{R,E,L\}}} (A,B,C) \quad \mbox{and} \quad x_{i(t)-1}(t) \in \bigcup_{\substack{(A,B,C)\in\{R,E,L\}\times\{R,E\}\\ \phantom{{}===={}}\times\{E,L\}}} (A,B,C),
\ees
recalling that, by Lemma \ref{lem:boundary-particles-fixed} and \eqref{eq:convention2}, $\Delta x_{i(t)}(t) = \Delta x_{i(t)-1}(t)$ if $i(t) =2$ and $\Delta x_{i(t)}(t) = \Delta x_{i(t)+1}(t)$ if $i(t) =N$. From Proposition \ref{prop:discrete-gradient-flow-well-posed} and Lemma \ref{lem:cases}, we then see that $\der \Delta x_{i(t)}/\der t(t) \geq 0$ for almost every $t\in[0,T]$. Therefore, by integrating between $0$ and $t$, we get
\be\label{eq:min-principle}
	\Delta x_{i(t)}(t) \geq \Delta x_{i(0)}(0) \quad \mbox{for all $t\in[0,T]$},
\ee 
which, with $\boldsymbol{x_N^0}$ being well-prepared for $\rho_0$ with constants $a_1$ and $a_2$, gives the result.

For the right-hand side inequality, we define $i\: [0,T]\to \{2,\dots,N\}$ such that $\Delta x_{i(t)}(t) = \max_{j\in\{2,\dots,N\}} \Delta x_j(t)$ for all $t\in[0,T]$, and then proceed similarly as above to get this time that $\der \Delta x_{i(t)}/\der t(t) \leq 0$ for almost every $t\in[0,T]$. Then, by integrating between $0$ and $t$,
\be\label{eq:max-principle}
	\Delta x_{i(t)}(t) \leq \Delta x_{i(0)}(0) \quad \mbox{for all $t\in[0,T]$},
\ee
which ends the proof since again $\boldsymbol{x_N^0}$ is well-prepared for $\rho_0$ with constants $a_1$ and $a_2$.
\end{proof}

	Lemma \ref{lem:bounds-distances} shows that under the hypotheses of Theorem \ref{thm:gradflow}, no particles of a discrete gradient flow solution can collide at any time in $[0,T]$. Equations \eqref{eq:min-principle} and \eqref{eq:max-principle} show respectively the existence of a weak minimum principle and a weak maximum principle at the discrete level.

\begin{lem}\label{lem:properties-interparticle}
	Let $(\mu_N)_{N\geq2}$ be as assumed in Theorem \ref{thm:gradflow} and suppose that $\liminf_{N\to\infty} g_N(\mu_N(t))$ is finite for all $t\in[0,T]$. Then, for all $t\in[0,T]$,
\be\label{eq:uniform-ratios}
	\max_{i\in\{2,\dots,N-1\}} \left|\frac{\Delta x_{i+1}(t)}{\Delta x_i(t)}-1\right| \xrightarrow[N\to\infty]{}0.
\ee
\end{lem}
\begin{proof}
	We omit in this proof the time dependences for simplicity, and let us use the notation $\psi_i:=1/\Delta x_i$ for all $i\in\{1,\dots,N\}$. By going through each case, Lemma \ref{lem:cases} yields
\bes
	|z_i| \geq \left|\psi_i - \psi_{i+1}\right| \quad \mbox{for all $i\in\{1,\dots,N\}$},
\ees
where $\bz:=(z_1,\dots,z_N) = \p_w^0\E(\bxN) \in \R_w^N$. Then, by \eqref{eq:sugEN},
\be\label{eq:discrete-fisher1}
	g_N(\mu_N)^2 = \frac1N \sum_{i=1}^N z_i^2 \geq \frac1N \sum_{i=1}^N \left(\psi_i  - \psi_{i+1}\right)^2.
\ee
The first and last terms of the sum above are equal to $0$ since $\Delta x_1 = \Delta x_2$ and $\Delta x_N = \Delta x_{N+1}$ by Lemma \ref{lem:boundary-particles-fixed} and therefore $\psi_1=\psi_2$ and $\psi_N=\psi_{N+1}$. It follows that
\bes
	g_N(\mu_N)^2 \geq \frac1N \sum_{i=2}^{N-1} \psi_{i+1}^2 \left( \frac{\psi_{i}}{\psi_{i+1}} - 1 \right)^2 = \frac1N \sum_{i=2}^{N-1} \frac{1}{\Delta x_{i+1}^2} \left( \frac{\Delta x_{i+1}}{\Delta x_i} - 1 \right)^2.
\ees
By Lemma \ref{lem:bounds-distances}, we know that $\Delta x_i \leq a_2/N$. Hence
\bes
	g_N(\mu_N)^2 \geq \frac{N}{a_2^2} \sum_{i=2}^{N-1} \left( \frac{\Delta x_{i+1}}{\Delta x_i} - 1 \right)^2.
\ees
Thus, for $\liminf_{N\to\infty} g_N(\mu_N)$ to be finite, \eqref{eq:uniform-ratios} must hold.
\end{proof}

\begin{rem}\label{rem:gaps}
	The quantity in \eqref{eq:uniform-ratios} controls how the total gap, i.e., the sum of all the gaps between non-overlapping intervals, behaves as $N$ goes to $\infty$. Indeed, a quick computation gives, omitting time dependence, that the total gap is
\begin{align*}
	\sum_{i=1}^{N-1} \left(\Delta x_{i+1} - \frac{r_i}{2} - \frac{r_{i+1}}{2}\right) &= \frac12 \sum_{i=1}^{N-1} \left(2\Delta x_{i+1} - \min(\Delta x_i, \Delta x_{i+1}) - \min(\Delta x_{i+1},\Delta x_{i+2})\right)\\
	& \leq \Delta x_2 + \Delta x_{N} + \sum_{i=2}^{N-2} \frac{\Delta x_{i+1}}{2} \left(2-\min\left(\frac{\Delta x_i}{\Delta x_{i+1}},1\right) - \min\left(1,\frac{\Delta x_{i+2}}{\Delta x_{i+1}}\right)\right)\\
	& \leq \frac{2a_2}{N} + a_2\max_{i\in\{2,\dots,N-1\}} \left|\frac{\Delta x_{i+1}}{\Delta x_i}-1\right|,
\end{align*}
thanks to Lemma \ref{lem:bounds-distances}. Lemma \ref{lem:properties-interparticle} therefore ensures that the total gap goes to 0 as $N$ increases and is controlled by $\max_{i\in\{2,\dots,N-1\}} |\Delta x_{i+1}/\Delta x_i-1|$.
\end{rem}

From now we assume $\liminf_{N\to\infty} g_N(\mu_N)$ is finite, or we are done, so that \eqref{eq:uniform-ratios} holds. We introduce an interpolation between particles. Let $\bxN \in [-\ell,\ell]_w^N$ be as in Theorem \ref{thm:gradflow}, and define
\be\label{eq:rho-tilde-l}
	\trho(x) := \frac{1/m_N}{N\Delta x_{i+1}^2} \left( \frac{x-x_i}{\sqrt{\Delta x_{i+1}}} + \frac{x_{i+1}-x}{\sqrt{\Delta x_i}} \right)^2 \quad \mbox{for $x \in [x_i,x_{i+1}]$, $i\in\{1,\dots,N-1\}$},
\ee
where we omit the time dependences, and where $m_N$ is the normalising constant given by
\bes
	m_N :=\frac{1}{3N}\sum_{i=1}^{N-1} \left(1+ \sqrt{\frac{\Delta x_{i+1}}{\Delta x_i}} + \frac{\Delta x_{i+1}}{\Delta x_i} \right).
\ees
We choose $\trho$ as in \eqref{eq:rho-tilde-l} because it belongs to $\Pac([-\ell,\ell])$, it has a well-defined Fisher information since it is continuous and $\trho>0$, and it gives rise to a simple computation of $g(\trho)$ in Section \ref{subsubsec:proof-slopes}. However, note that choosing $\trho$ to be linear would still work very similarly. We can show that $\trho$ is a good narrow approximation of our limiting measure $\rho$. 

\begin{lem} \label{lem:weak}
	Let $\rho$ be as in Theorem \ref{thm:gradflow}. Then $\trho(t) \wto \rho(t)$ narrowly as $N\to\infty$ for all $t\in[0,T]$. 
\end{lem}
\begin{proof}
	For simplicity, we omit the time dependences throughout this proof. Let $\varphi \in C_\mt{b}([-\ell,\ell])$ be Lipschitz with constant $L>0$. Write $\nu_N := m_N \trho$, so that $\nu_N(x_i) = 1/(N\Delta x_i)$ for all $i \in\{1,\dots,N\}$, and $\phi_N:= \big|\int_{-\ell}^\ell \varphi(x) \nu_N(x) \d x - \int_{-\ell}^\ell \varphi(x) \d\mu_N(x)\big|$, and compute
\begin{align*}
	\phi_N &= \left| \sum_{i=2}^N \int_{x_{i-1}}^{x_i} \varphi(x) \nu_N(x) \d x - \frac1N \sum_{i=1}^N \varphi(x_i)\right| \leq \sum_{i=2}^N \int_{x_{i-1}}^{x_i} \left|\varphi(x) \nu_N(x) - \frac{\varphi(x_i)}{N\Delta x_i} \right|\d x + \frac{\left|\varphi(x_1)\right|}{N}\\
	&\leq \sum_{i=2}^N \int_{x_{i-1}}^{x_i} \left|\varphi(x) \nu_N(x) - \varphi(x_i)\nu_N(x_i) \right|\d x + \frac{\norm{\varphi}_\infty}{N}\\
	&\leq \norm{\varphi}_\infty \sum_{i=2}^N \int_{x_{i-1}}^{x_i} |\nu_N(x) - \nu_N(x_i)| \d x + \frac1N \sum_{i=2}^N \int_{x_{i-1}}^{x_i} \frac{|\varphi(x) - \varphi(x_i)|}{\Delta x_i} \d x + \frac{\norm{\varphi}_\infty}{N}.\\
	&\leq \frac{2\norm{\varphi}_\infty}{N} \sum_{i=2}^N \frac{\left| \Delta x_i^{-1/2} - \Delta x_{i-1}^{-1/2} \right|}{\min(\sqrt{\Delta x_{i-1}},\sqrt{\Delta x_i})} \int_{x_{i-1}}^{x_i} \frac{|x-x_i|}{\Delta x_i} \d x+ \frac LN \sum_{i=2}^N \int_{x_{i-1}}^{x_i} \frac{|x-x_i|}{\Delta x_i} \d x + \frac{\norm{\varphi}_\infty}{N}
\end{align*}
By Lemma \ref{lem:bounds-distances}, we have
\bes
	\phi_N  \leq \frac{2a_2\norm{\varphi}_\infty}{a_1} \max_{i\in\{2,\dots,N\}} \left| 1- \sqrt{\frac{\Delta x_i}{\Delta x_{i-1}}} \right| + \frac{La_2+ \norm{\varphi}_\infty}{N} \xrightarrow[N\to\infty]{}0,
\ees
by \eqref{eq:uniform-ratios} and the fact that $\Delta x_1=\Delta x_2$. This shows that $\nu_N \wto \rho$ narrowly as $N\to \infty$ and $\nu_N([-\ell,\ell]) \to \rho([-\ell,\ell]) =1$ as $N\to \infty$, which yields $m_N \to 1$ and thus $\trho \wto \rho$ narrowly.
\end{proof}

\subsubsection{Proof of Lemma \ref{lem:gradflow-cond2}}\label{subsubsec:proof-slopes}

We omit time dependence. The proof of Lemma \ref{lem:gradflow-cond2} now reduces to showing that $\trho$ gives rise also to a good estimate of $g_N(\mu_N)$ and $g(\rho)$, that is
\be\label{eq:intermediate}
	\liminf_{N\to\infty} g_N(\mu_N) \geq \liminf_{N\to\infty} g(\trho) \geq g(\rho),
\ee
where $(\trho)_{N\geq2}$ is the sequence associated to $(\mu_N)_{N\geq 2}$ defined as in \eqref{eq:rho-tilde-l}. For the right-hand inequality, this is an immediate consequence of Lemma \ref{lem:weak} and the narrow lower semi-continuity of $g$, see \cite[Corollary 2.4.10]{AGS}. 
	
	For the specific case of the heat equation, the Fisher information \eqref{eq:fisher} is $I(\rho) = \int_{-\ell}^\ell \rho'(x)^2/\rho(x)\d x$ if $\rho\in W^{1,1}([-\ell,\ell])$. Therefore, for the left-hand inequality, we can compute, by Lemma \ref{lem:sugE},
\begin{align*}
	g(\trho)^2 &= \frac{1/m_N}{N} \sum_{i=1}^{N-1} \frac{4}{\Delta x_{i+1}^2} \int_{x_i}^{x_{i+1}} \left(\frac{1}{\sqrt{\Delta x_{i+1}}} - \frac{1}{\sqrt{\Delta x_i}} \right)^2 \d x = \frac{4/m_N}{N} \sum_{i=1}^{N-1} \frac{\left(\Delta x_{i+1}^{-1} - \Delta x_i^{-1} \right)^2}{\left( 1+\sqrt{\Delta x_{i+1}/\Delta x_i} \right)^2}.
\end{align*}
Let $0< \epsilon < 4$. By \eqref{eq:uniform-ratios}, we have $\Delta x_{i+1}/\Delta x_i \to 1$ as $N\to\infty$, for all $i \in\{1,\dots,N\}$, uniformly in $i$, and, by the proof of Lemma \ref{lem:weak}, $m_N \to 1$. Therefore, there exists $N(\epsilon)$ large enough such that $(1/m_N)/(1+\sqrt{\Delta x_{i+1}/\Delta x_i})^2 < 1/(4 - \epsilon)$ for all $N>N(\epsilon)$ and $i\in\{1,\dots,N-1\}$. Thus, for any such $N$, the equality above becomes
\bes
	g(\trho)^2 \leq \frac{4}{N(4-\epsilon)} \sum_{i=1}^{N-1} \left(\frac{1}{\Delta x_{i+1}} - \frac{1}{\Delta x_i} \right)^2 \leq \frac{4g_N(\mu_N)^2}{(4-\epsilon)}
\ees
by \eqref{eq:discrete-fisher1}. Then, taking the limits $N\to\infty$ and $\epsilon \to0$ in this order, we get the result.\qed

\subsection{General density of internal energy}\label{subsec:general-internal}

We want now to extend Section \ref{subsec:he} to general densities of internal energy $H$ satisfying \ref{hyp:general1} and \ref{hyp:general2}. Note that \ref{hyp:general2} implies $h''(x) > 0$ for all $x\in(0,\infty)$.

\subsubsection{Preliminaries}
We compute the local slope $g_N$ of $\E$. Equation \eqref{eq:sugEN} still holds and we can characterise the subdifferential of $\wtE$ in the same fashion as for the heat equation. In fact, Lemma \ref{lem:subdiff-EN} is still true, where $\psi_i$ takes now the general form $\psi_i:= -Nh'(N\Delta x_i)$ for all $i\in\{1,\dots,N\}$. Note that since $h''>0$ by assumption in Theorem \ref{thm:gradflow}, $h'$ is increasing and therefore $\psi_i > \psi_{i+1}$ if $\Delta x_i > \Delta x_{i+1}$, and vice versa; also, since $h$ is non-increasing, $\psi_i \geq 0$. The minimal norm element is still given by Lemma \ref{lem:cases}, where $\psi_i$ takes its general form. Lemma \ref{lem:boundary-particles-fixed} still holds by the monotonicity property of $(\psi_i)_i$. Lemma \ref{lem:bounds-distances} remains unchanged and can be proved in the same manner, again by monotonicity of $(\psi_i)_i$. Lemma \ref{lem:properties-interparticle} still holds; it is proved similarly as for the heat equation since $h'$ is increasing and non-positive, and, on top of \eqref{eq:uniform-ratios}, the proof also gives
\be\label{eq:uniform-ratios-h}
	\max_{i\in\{2,\dots,N-1\}} \left|\frac{\psi_{i}(t)}{\psi_{i+1}(t)}-1\right| \xrightarrow[N\to\infty]{}0 \quad \mbox{for all $t\in[0,T]$}.
\ee

We need to generalise the definition of the interpolation $\trho$ in \eqref{eq:rho-tilde-l} used in Section \ref{subsubsec:proof-slopes} for the heat equation. To this end we introduce the function $\psi\: (0,\infty) \to [0,\infty)$ by
	\bes
		\psi(x) = -h'(x) \quad \mbox{for all $x\in(0,\infty)$.}
	\ees
Clearly $\psi_i = N\psi(N\Delta x_i)$ for all $i\in\{1,\dots,N\}$, and, since $h'$ is increasing, $\psi$ is decreasing and therefore invertible. Define, omitting the time dependences,
\be\label{eq:rho-tilde-general}
	\trho(x) := \frac{1/m_N}{\psi^{-1}(p_{i,k}(x))} \quad \mbox{for $x \in [x_i,x_{i+1}]$, $i\in\{1,\dots,N-1\}$},
\ee
where $m_N = \int_{-\ell}^\ell 1/\psi^{-1}(p_{i,k}(x)) \d x$ makes $\trho$ belong to $\Pac([-\ell,\ell])$, and where, for any $i\in\{1,\dots,N-1\}$ and $k\in\N$, $p_{i,k}\:[x_i,x_{i+1}] \to (0,\infty)$ is the monotone function
\bes
	p_{i,k}(x) = \left( \frac{x-x_i}{\Delta x_{i+1}} \left(\frac{\psi_{i+1}}{N}\right)^{1/k} + \frac{x_{i+1}-x}{\Delta x_{i+1}} \left(\frac{\psi_{i}}{N}\right)^{1/k}  \right)^k \quad \mbox{for all $x\in[x_i,x_{i+1}]$}.
\ees
Obviously $p_{i,k}(x_i) = \psi_i/N$ and $p_{i,k}(x_{i+1}) = \psi_{i+1}/N$, $m_N\trho(x_i) = 1/(N\Delta x_i)$, and,
\be\label{eq:psi-N}
	N\min(\Delta x_i,\Delta x_{i+1}) \leq \psi^{-1}(p_{i,k}(x)) \leq N\max(\Delta x_i,\Delta x_{i+1}) \quad \mbox{for $x \in [x_i,x_{i+1}]$},
\ee
which, by Lemma \ref{lem:bounds-distances}, yields
\be\label{eq:psi}
	a_1 \leq \psi^{-1}(p_{i,k}(x)) \leq a_2 \quad \mbox{for $x \in [x_i,x_{i+1}]$}.
\ee
In the following we choose the interpolation functions $(p_{i,k})_{i\in\{1,\dots,N-1\}}$ to be linear, i.e., $k=1$; in this case we simply write $p_i = p_{i,1}$. This is only a choice that makes the computations below simpler; any other $k\in\N$ works in a very similar manner. Note that in the case of the heat equation, we choose $k=2$, see \eqref{eq:rho-tilde-l}, as that particular choice makes the calculation of $g(\trho)$ much easier in Section \ref{subsubsec:proof-slopes} because of some cancellations. These simplifications do not hold anymore in this general setting, and we thus pick the simplest interpolations, which are the linear ones. Let us point out that for $\trho$ as in \eqref{eq:rho-tilde-general} the Fisher information is well-defined since $\trho$ is continuous and $\trho >0$. We give here the proof of Lemma \ref{lem:weak} adapted to this general setting.

\begin{lem} \label{lem:weak-general}
	Let $\rho$ be as in Theorem \ref{thm:gradflow}. Then $\trho(t) \wto \rho(t)$ narrowly as $N\to\infty$ for all $t\in[0,T]$. 
\end{lem}
\begin{proof}
	We use the same notation as in the proof of Lemma \ref{lem:weak}. We have
\begin{align*}
	\phi_N &\leq \norm{\varphi}_\infty \sum_{i=2}^N \int_{x_{i-1}}^{x_i} |\nu_N(x) - \nu_N(x_i)| \d x + \frac{La_2+ \norm{\varphi}_\infty}{N}\\
	&\leq \norm{\varphi}_\infty\sum_{i=2}^N \frac{\left| p_{i-1}'(\xi_i(x)) \right|}{\left|\psi'\circ\psi^{-1}(p_{i-1}(\xi_i(x)))\right| [\psi^{-1}(p_{i-1}(\xi_i(x)))]^2} \int_{x_{i-1}}^{x_i} |x-x_i| \d x + \frac{La_2+ \norm{\varphi}_\infty}{N},
\end{align*}
where $\xi_i\: [x_{i-1},x_i] \to (x_{i-1},x_i)$ for all $i\in\{2,\dots,N\}$ are functions stemming from the mean-value theorem.
By Lemma \ref{lem:bounds-distances}, \eqref{eq:psi}, the smoothness of $\psi'$ and linearity of $p_i$, we have
\bes
	\phi_N  \leq \frac{a_2\psi(a_1)\norm{\varphi}_\infty}{a_1^2 \min_{x\in[a_1,a_2]} |\psi'(x)|} \max_{i\in\{2,\dots,N\}} \left| 1- \frac{\psi_{i-1}}{\psi_{i}}\right| + \frac{La_2+ \norm{\varphi}_\infty}{N} \xrightarrow[N\to\infty]{}0,
\ees
by \eqref{eq:uniform-ratios-h} and the fact that $\Delta x_1=\Delta x_2$ and so $\psi_1 = \psi_2$. Note that $\min_{x\in[a_1,a_2]} |\psi'(x)| >0$ since $h''>0$. As in the proof of Lemma \ref{lem:weak}, this shows $m_N \to 1$ and $\trho \wto \rho$ narrowly.
\end{proof}

\subsubsection{Proof of Lemma \ref{lem:gradflow-cond2}}\label{subsubsec:proof-slopes-general}

We omit time dependence. We want to show that \eqref{eq:intermediate} is still true. For the right-hand inequality, this is again an immediate consequence of Lemma \ref{lem:weak-general} and the narrow lower semi-continuity of the local slope $g$.
	
For the left-hand inequality, let us write $\nu_N:=m_N \trho$. By abuse, we can compute the Fisher information at $\nu_N$, even if $\nu_N$ does not necessarily have unit mass. It is easy to check that, for all $i\in\{1,\dots,N-1\}$, the integrand of the Fisher information \eqref{eq:fisher} for $\nu_N$ is
\bes
	\psi'\left(\frac{1}{\nu_N(x)}\right)^2\frac{\nu_N'(x)^2}{\nu_N(x)^5} = p_i'(x)^2 \psi^{-1}(p_i(x)) \quad \mbox{for all $x\in[x_i,x_{i+1}]$}.
\ees
Therefore, by Lemma \ref{lem:sugE} and using \eqref{eq:psi-N} and the linearity of $p_i$,
\bes
	g(\nu_N)^2 = \frac{1}{N^2} \sum_{i=1}^{N-1} \int_{x_{i}}^{x_{i+1}} (\psi_{i+1} - \psi_{i})^2 \frac{\psi^{-1}(p_i(x))}{\Delta x_{i+1}^2} \d x \leq \frac{1}{N} \sum_{i=1}^{N-1}  (\psi_{i+1} - \psi_{i})^2 \max\left(1,\frac{\Delta x_i}{\Delta x_{i+1}}\right).
\ees
Let $\epsilon >0$. By \eqref{eq:uniform-ratios} we have $\Delta x_i/\Delta x_{i+1} \to 1$ as $N\to\infty$, for all $i \in\{1,\dots,N-1\}$, uniformly in $i$. Therefore, there exists $N(\epsilon)$ large enough such that $\min(1,\Delta x_i/\Delta x_{i+1})<1+\epsilon$ for all $N>N(\epsilon)$ and $i\in\{1,\dots,N-1\}$. For such $N$ we then get, by \eqref{eq:discrete-fisher1},
\bes
	g(\nu_N)^2 \leq \frac{1+\epsilon}{N} \sum_{i=1}^{N-1} (\psi_{i+1} - \psi_{i})^2 \leq (1+\epsilon)g_N(\mu_N)^2.
\ees
By taking the limits $N\to\infty$ and $\epsilon\to0$ in this order, we get
\bes
	\liminf_{N\to\infty} g(\nu_N) \leq \liminf_{N\to\infty} g_N(\mu_N).
\ees
In order to conclude, we only need to show that $\liminf_{N\to\infty} g(\nu_N) \geq \liminf_{N\to\infty} g(\trho)$. Compute
\begin{align*}
	g(\nu_N)^2 &= g(m_N\trho) = m_N^3 \int_{-\ell}^\ell \trho'(x)^2 H''(m_N\trho(x))^2\trho(x) \d x \geq m_N^3 f(m_N)^2 g(\trho)^2,
\end{align*}
where $f$ is as in \ref{hyp:general2}, showing the result since $m_N\to1$ by the proof of Lemma \ref{lem:weak-general}.\qed

\section{Extensions}\label{sec:extension-R}

In this section we discuss extensions of the main theorem to the whole real line and general weights. We intentionally give no computations as we only see this section as an outlook for a possible future work.

\subsection{Extension to the whole line}\label{subsec:extension-whole-line}

The extension of the convergence part in the main theorem to the whole real line, i.e., $\O = \R$, is not an easy task. The only part of the proof that needs to be adapted is Section \ref{sec:conditions}, that is, the proof of \ref{cond:slopes} on the lower semi-continuity of the local slopes. We now point out the main arguments of Section \ref{sec:conditions} that need adapting in order to fit the whole-line situation, and we explain where our approach fails.

We first discuss the case of the heat equation given in Section \ref{subsec:he}. Note that the computation of the discrete local slope, given through Equation \eqref{eq:sugEN} to Lemma \ref{lem:cases}, remains unchanged if $\O=\R$. Lemma \ref{lem:bounds-distances} needs to be changed; indeed, the upper bound on the inter-particle distances cannot be preserved, i.e., there is no weak maximum principle at the discrete level if no boundary conditions are imposed. This is because when a maximal inter-particle interval happens to be, for instance, the leftmost one, it can actually get even wider since no (fictitious) particle is on its left to prevent it from moving leftwards faster than its right neighbour; this, in turn, is a consequence of the fact that the speed of propagation for the heat equation is infinite at the continuum level. However, one can still prove that inter-particle distances cannot grow too much. In fact, one can show
\be\label{eq:max}
	a_1N^{-1}\leq \Delta x_i(t) \leq c_N(T) \quad \mbox{for all $i\in\{2,\dots,N\}$, for all $t\in[0,T]$,}
\ee
where $c_N(T) \to \sqrt{2T}$ as $N\to\infty$. This lack of weak maximum principle yields the failure of Lemma \ref{lem:properties-interparticle}, which no longer gives us the uniform behaviour of the inter-particle distances as $N$ increases. This is the crucial fact that makes our approach fail if no boundary conditions are applied. Indeed, we cannot hope to get convergence if we do not have a proper control on how the gaps between the discretisation intervals decrease as $N$ increases, see Remark \ref{rem:gaps}; then, our proof of \eqref{eq:intermediate} fails. The interpolation \eqref{eq:rho-tilde-l} remains almost untouched except at the boundary particles where its Fisher information should be taken so as to match the discrete local slope. It is then still a good narrow approximation of $\rho$, as in Lemma \ref{lem:weak}, if the uniform behaviour \eqref{eq:uniform-ratios} of the inter-particle distances is assumed.

For the case of the general density of internal energy given in Section \ref{subsec:general-internal}, most of the above remarks still hold. However, instead of \eqref{eq:max}, one has
\be\label{eq:max-general}
	a_1N^{-1}\leq \Delta x_i(t) \leq c_N(T) \quad \mbox{for all $i\in\{2,\dots,N\}$, for all $t\in[0,T]$,}
\ee
with
\be\label{eq:Ph}
	c_N(T) = N^{-1} \Psi^{-1}\left( N^2T + \Psi(a_2) \right),
\ee
where $\Psi$ is any antiderivative of $1/\psi$. As for the heat equation case, this lack of ``good'' control on the inter-particle intervals leads to the failure of our proof of \eqref{eq:intermediate}. Here, it is actually not even clear to us whether the interpolation \eqref{eq:rho-tilde-general} is still a good narrow approximation of $\rho$, as it is in Lemma \ref{lem:weak-general}, even if the uniform behaviour \eqref{eq:uniform-ratios} is assumed.

Interestingly, for the case of the porous medium equation ($H(x) = x^{m-1}/(m-1)$ with $m>1$), \eqref{eq:max-general} implies, contrary to the case of the heat equation, that the maximal inter-particle distance actually decreases as $N$ increases, although the weak maximum principle is still not preserved. Indeed, for the porous medium case, one can pick $\Psi(x) = x^{m+1}/(m+1)$, so that $c_{N}(T) \sim ((m+1)T/N^{m-1})^{1/(m+1)}$ as $N\to\infty$ by \eqref{eq:Ph}. This stems from the fact that, at the continuum level, the solution to the porous medium equation is compactly supported at all times, which ensures some compactness at the discrete level as well. Unfortunately, this decreasing behaviour of the inter-particle intervals is still not enough to get their uniform behaviour, and therefore a control on the gaps, as $N$ increases; indeed the support of the solution still spreads (albeit asymptotically) to the whole real line.

Let us illustrate, using the case of the heat equation, the fact that when $\O=\R$ we cannot expect to have a weak maximum principle and that therefore the uniform behaviour of the inter-particle distance should be hoped to come from somewhere else. Adapting the proof of Lemma \ref{lem:properties-interparticle} to the whole-line setting ($\Delta x_1 = \Delta x_{N+1} = \infty$, by \eqref{eq:convention1}) gives that $\Delta x_2 \geq C/\sqrt{N}$ and $\Delta x_N \geq C/\sqrt{N}$ for some $C>0$. This contradicts the fact that the inter-particle distances are of order $1/N$ at all times, as they are initially. However, this does not mean that the proper behaviour of the inter-particle gaps cannot be obtained by other means. Indeed, it could still follow from inter-particle distances of order $1/\sqrt{N}$, rather than $1/N$, in a way we do not know.

In view of these remarks, the question of extending Theorem \ref{thm:gradflow} to the whole real line is still open. However, we can conclude by stating a first result in that direction for the heat equation, where the uniform behaviour of the inter-particle distances is assumed.
\begin{thm}\label{thm:gradflow-R}
	Let $H$ be the density of internal energy for the heat equation. Suppose that $\mu_N \in AC^2([0,T],\A_{N,w}(\O))$, with particles $\bxN \in AC^2([0,T],\O_w^N)$, is a discrete gradient flow solution with initial condition $\mu_N^0 \in \A_{N,w}(\O)$, with particles $\boldsymbol{x_N^0} \in \O_w^N$. Let $\rho_0 \in \G(\O)$ and assume that $(\mu_N^0)_{N\geq2}$ is well-prepared for $\rho_0$ according to Definition \ref{defn:initial-set}. Then there exists $\rho\in AC^2([0,T],\P_2(\O))$ such that $\mu_N(t)\wto \rho(t)$ narrowly as $N\to\infty$ for all $t \in [0,T]$. Moreover, if
\bes
	\max_{i\in\{2,\dots,N-1\}} \left|\frac{\Delta x_{i+1}(t)}{\Delta x_i(t)}-1\right| \xrightarrow[N\to\infty]{}0 \quad \mbox{for all $t\in[0,T]$},
\ees
then $\rho$ is the continuum gradient flow solution associated to \eqref{eq:pde} in the sense of Definition \ref{defn:continuum-gradient-flow}, and \eqref{eq:limits} holds.
\end{thm}

\subsection{Extension to general weights}

In the numerical tests performed in the companion paper \cite{CHPW}, the weights of the particles are allowed to be non-equal. Extending our proof of convergence to particles with general (non-equal) weights is thus a natural question. First, note that the whole discretisation of Section \ref{subsec:particle-method} can be generalised to weights $w = (w_1,\dots,w_N) \subset (0,1)^N$ with $\sum_{i=1}^N w_i = 1$ and $\max_{i\in\{1,\dots,N\}} w_i\to 0$ as $N\to\infty$, see \cite{CHPW}. 

Every argument in Sections \ref{sec:compactness} and \ref{subsec:energy} still holds for such general weights by simple changes---the conditions \ref{cond:md} and \ref{cond:liminf} are thus still true. Section \ref{subsec:gamma} on the $\Gamma$-convergence of the discrete energy needs the re-definition of the notion of well-preparedness; indeed, the bound condition in Definition \ref{defn:initial-set} now has to be $a_1w_i \leq \Delta x_i \leq a_2w_i$ for all $i\in\{2,\dots,N\}$ and all $N\geq2$.
Then, Lemma \ref{lem:limsup-continuous} can be easily adapted to the general-weight setting by rewriting the well-prepared sequence \eqref{eq:initial-set} as
\bes
	\begin{cases} x_1 = \Phi(0),\\ x_i = \Phi\left(\sum_{j=1}^i w_j \right) & \mbox{for $i \in \{2,\dots,N\}$}. \end{cases}
\ees
Lemma \ref{lem:mollifier} does not need any changes. 

Section \ref{sec:conditions} is the part of the proof of the main theorem that needs the most delicate adapting. It is still not clear to us how the proof of \ref{cond:slopes} can be extended to general weights. In fact, Lemma \ref{lem:subdiff-EN} holds with little changes; however, we are no longer able to compute the element of minimal norm as simply as in Lemma \ref{lem:cases}, which subsequently does not allow us to conclude. We believe that a uniform control on the weights, such as
\bes
	\max_{i\in\{1,\dots,N-1\}} \left| \frac{w_{i+1}}{w_i} - 1\right| \xrightarrow[N\to\infty]{}0,
\ees
could be of help, although this would deserve more investigation, which we leave to future work.

\subsection*{Acknowledgements}

JAC is supported by the Royal Society through a Wolfson Research Merit Award. FSP is grateful to Daniel Matthes for insightful discussions. PS gratefully acknowledges the support of the National Science Foundation through DMS-1362879. PS is also supported by The Lady Davis Fellowship from the Technion. GW is supported by ISF grant 998/5.

\bibliography{gamma_convergence}
\bibliographystyle{abbrv}

\end{document}